\documentclass[oneside, 11pt]{amsart} 

\usepackage{amsfonts}
\usepackage{amssymb}
\usepackage{latexsym}
\usepackage{graphics}
\usepackage[2emode]{psfrag}
\usepackage{amsthm}
\usepackage{amsmath}
\usepackage[all]{xy}
             
\begin{document}  

\newcommand{\norm}[1]{\| #1 \|}
\def\N{\mathbb N}
\def\Z{\mathbb Z}
\def\Q{\mathbb Q}
\def\mod{\textit{\emph{~mod~}}}
\def\R{\mathcal R}
\def\S{\mathcal S}
\def\*  C{{*  \mathcal C}} 
\def\C{\mathcal C}
\def\D{\mathcal D}
\def\J{\mathcal J}
\def\M{\mathcal M}
\def\T{\mathcal T}          

\newcommand{\Hom}{{\rm Hom}}
\newcommand{\End}{{\rm End}}
\newcommand{\Ext}{{\rm Ext}}
\newcommand{\Mor}{{\rm Mor}\,}
\newcommand{\Aut}{{\rm Aut}\,}
\newcommand{\Hopf}{{\rm Hopf}\,}
\newcommand{\Ann}{{\rm Ann}\,}
\newcommand{\Coker}{{\rm Coker}\,}
\newcommand{\im}{{\rm Im}\,}
\newcommand{\coim}{{\rm Coim}\,}
\newcommand{\Trace}{{\rm Trace}\,}
\newcommand{\Char}{{\rm Char}\,}
\newcommand{\Mod}{{\rm mod}}
\newcommand{\Spec}{{\rm Spec}\,}
\newcommand{\sgn}{{\rm sgn}\,}
\newcommand{\Id}{{\rm Id}\,}
\newcommand{\Com}{{\rm Com}\,}
\newcommand{\codim}{{\rm codim}}
\newcommand{\Mat}{{\rm Mat}}
\newcommand{\can}{{\rm can}}
\newcommand{\sign}{{\rm sign}}
\newcommand{\kar}{{\rm kar}}
\newcommand{\rad}{{\rm rad}}
\newcommand{\Ker}{{\rm Ker}}
\newcommand{\Inv}{{\rm Inv}}
\newcommand{\Den}{{\rm Den}}
\newcommand{\Img}{{\rm Im}}
\newcommand{\ra}{\rightarrow}
\newcommand{\rs}{\rightsquigarrow}

\def\lan{\langle}
\def\ran{\rangle}
\def\ot{\otimes}

\def\id{{\small \textit{\emph{1}}}\!\!1}    
\def\To{{\multimap\!\to}}
\def\bigperp{{\LARGE\textrm{$\perp$}}} 
\newcommand{\QED}{\hspace{\stretch{1}}
\makebox[0mm][r]{$\Box$}\\}

\def\RR{{\mathbb R}}
\def\FF{{\mathbb F}}
\def\NN{{\mathbb N}}
\def\CC{{\mathbb C}}
\def\DD{{\mathbb D}}
\def\ZZ{{\mathbb Z}}
\def\QQ{{\mathbb Q}}
\def\HH{{\mathbb H}}
\def\units{{\mathbb G}_m}
\def\GG{{\mathbb G}}
\def\EE{{\mathbb E}}
\def\FF{{\mathbb F}}
\def\rightact{\hbox{$\leftharpoonup$}}
\def\leftact{\hbox{$\rightharpoonup$}}

\newcommand{\Aa}{\mathcal{A}}
\newcommand{\Bb}{\mathcal{B}}
\newcommand{\Cc}{\mathcal{C}}
\newcommand{\Dd}{\mathcal{D}}
\newcommand{\Ee}{\mathcal{E}}
\newcommand{\Ff}{\mathcal{F}}
\newcommand{\Hh}{\mathcal{H}}
\newcommand{\Ii}{\mathcal{I}}
\newcommand{\Mm}{\mathcal{M}}
\newcommand{\Pp}{\mathcal{P}}
\newcommand{\Rr}{\mathcal{R}}
\def\*  C{{}*  \hspace*  {-1pt}{\Cc}}

\def\text#1{{\rm {\rm #1}}}

\def\smashco{\mathrel>\joinrel\mathrel\triangleleft}
\def\cosmash{\mathrel\triangleright\joinrel\mathrel<}

\def\Nat{\dul{\rm Nat}}

\renewcommand{\subjclassname}{\textup AMS Mathematics Subject
     Classification (2000)}

\newtheorem{prop}{Proposition}[section] 
\newtheorem{lemma}[prop]{Lemma}
\newtheorem{cor}[prop]{Corollary}
\newtheorem{theo}[prop]{Theorem}

\theoremstyle{definition}
\newtheorem{Def}[prop]{Definition}
\newtheorem{ex}[prop]{Example}
\newtheorem{exs}[prop]{Examples}
\newtheorem{Not}[prop]{Notation}
\newtheorem{Ax}[prop]{Axiom}
\newtheorem{rems}[prop]{Remarks}
\newtheorem{rem}[prop]{Remark}

\def\smashco{\mathrel>\joinrel\mathrel\triangleleft}
\def\curlarrow{\mathrel\sim\joinrel\mathrel>}

\title{Involutive filters of pseudo-hoops}
%\date{2010.09.15}
\author{Lavinia Corina Ciungu}

\begin{abstract} 
In this this paper we introduce the notion of involutive filters of pseudo-hoops, and we emphasize their 
role in the probability theory on these structures. 
A characterization of involutive pseudo-hoops is given and their properties are investigated. 
We give characterizations of involutive filters of a bounded pseudo-hoop and we prove that in the case of 
bounded Wajsberg pseudo-hoops the notions of fantastic and involutive filters coincide. 
One of main results consists of proving that a normal filter $F$ of a bounded pseudo-hoop $A$ is involutive 
if and only if $A/F$ is an involutive pseudo-hoop. It is also proved that any Boolean filter of a bounded Wajsberg pseudo-hoop is involutive. The notions of state operators and state-morphism operators on pseudo-hoops are 
introduced and the relationship between these operators are investigated. 
For a bounded Wajsberg pseudo-hoop we prove that the kernel of any state operator is an involutive filter. \\


\end{abstract}

\maketitle

\section{Introduction}

Many information processing branches are based on the non-classical logics and deal with uncertainty information  (fuzziness, randomness, vagueness, etc.). There is a strong motivation to revise the classical probability theory 
and to introduce more general probability models based on non-classical logics. 
Different probabilistic models have been constructed on algebras of fuzzy logics: states, generalized
states, internal states, state-morphism operators, measures.
Filters on non-commutative multiple-valued algebras proved to play an important role for studying the existence  
of probabilistic models on these structures (\cite{DvRa1}, \cite{DvRa2}, \cite{Dvu3}, \cite{Rac1}, \cite{Geo1}, \cite{DiDv1}, \cite{DiDv1-e}, \cite{Ciu6}) and to investigate their main properties (\cite{DiDv2}, \cite{DiDv3}, \cite{Dvu5}, \cite{DvRa3}, \cite{Dvu103}). 
Pseudo-hoops were introduced in \cite{Geo16} as a generalization of hoops which were originally defined and studied   by Bosbach in \cite{Bos1} and \cite{Bos2} under the name of complementary semigroups. 
It was proved that a pseudo-hoop has the pseudo-divisibility condition and it is a meet-semilattice, so a bounded R$\ell$-monoid can be viewed as a bounded pseudo-hoop together with the join-semilattice property. 
In other words, a pseudo-hoop is a meet-semilattice ordered residuated, integral and divisible monoid. 
The pseudo-hoops have been intensively studied in (\cite{Dvu4}, \cite{DvGiKo}, \cite{DvKu}, \cite{Ciu2}, \cite{Alav1}).
In the last years many works were dedicated to the study of probabilities theories on hoops and pseudo-hoops 
(\cite{Bor2}, \cite{Ciu10}, \cite{Ciu20}, \cite{Ciu22}, \cite{Ciu23}). \\

In this paper we show that the commutative property plays an important role in probabilities theory on 
pseudo-hoops. Important results on probabilistic models on algebras of non-classical logic
have been proved based on involutive filters. 
We give a characterization of bounded Wajsberg pseudo-hoops and we recall some properties 
of these structures. We define the notion of involutive pseudo-hoop, we give a characterization of involutive pseudo-hoops and we investigate their properties. 
We introduce the notion of a normal pseudo-hoop and we prove that the set of all involutive elements of a 
normal pseudo-hoop $A$ is a subalgebra of $A$. 
We define the notion of an Archimedean pseudo-hoop and we prove that a pseudo-hoop is Archimedean if and only 
if it is a linearly ordered Wajsberg pseudo-hoop. 
As a consequence, any simple pseudo-hoop is a linearly ordered Wajsberg pseudo-hoop.
The notion of an involutive filter of a bounded pseudo-hoop 
is defined and it is proved that the kernel of a Bosbach state (state-measure, type II state operator) on 
pseudo-hoops is an involutive filter.
If $\Den(A)$ is the set of all dense elements of a good pseudo-hoop $A$, we show that $\Den(A)$ is an 
involutive filter of $A$, and any filter of $A$ containing $\Den(A)$ is an involutive filter. 
One of main results consists of proving that a normal filter $F$ of a bounded pseudo-hoop $A$ is involutive if and 
only if $A/F$ is an involutive pseudo-hoop. 
We introduce the notion of a fantastic filter of a pseudo-hoop $A$ and we prove that a normal filter of $A$ 
is fantastic if and only if $A/F$ is a Wajsberg pseudo-hoop. 
We give characterizations of involutive filters of a bounded pseudo-hoop and we prove that in the case of 
bounded Wajsberg pseudo-hoops the notions of fantastic and involutive filters coincide. 
It is also proved that any Boolean filter of a bounded Wajsberg pseudo-hoop is involutive. 
The concept of a state pseudo-hoop have been developed in two directions: \\
\indent 
$-$ By generalization of state operators from bounded R$\ell$-monoids (\cite{DvRa3}, \cite{DvRa4}) to the case of 
bounded pseudo-hoops (\cite{Ciu22}, \cite{Ciu23}). \\
\indent
$-$ By defining the notion of state operators on hoops (\cite{Bor2}) as a particular case of state operators on 
BCK-algebras (\cite{Bor1}). \\
We unify the two concepts of state operators and we introduce a more general notion of state operators on pseudo-hoops. 
More precisely we define three types of state operators on pseudo-hoops: type I and type II as generalization of 
state operators on hoops from \cite{Bor2}, and type III as generalization of state operators on bounded pseudo-hoops from \cite{Ciu22}, \cite{Ciu23}. 
We prove that a pseudo-hoop is Wajsberg if and only if the type I and type II state operators coincide. 
For the case of a bounded pseudo-hoop it is proved that the kernel of a type II state operator is an involutive filter. Moreover, for a bounded Wajsberg pseudo-hoop the kernel of any type of state operator is an involutive filter. 
As main results we show that any type II state operator is a type III state operator, and in the case of bounded 
Wajsberg pseudo-hoops any type I state operator is a type III state operator. 
If the kernel of a type II state operator $\mu$ is a normal filter, then it is proved that $\mu$ is also a type I 
state operator. \\
We define the notion of a state-morphism operator on pseudo-hoops and we prove that any state-morphism operator is 
a type I and type III state operator. For the case of an idempotent pseudo-hoop it is proved that any type II or 
type III state operator is a state-morphism operator, while for a bounded idempotent Wajsberg pseudo-hoop any type I 
state operator is also a state-morphism. 
Another main result consists of proving that any state-morphism on the subalgebra of involutive elements of a 
bounded idempotent pseudo-hoop $A$ can be extended to a state-morphism on $A$.

$\vspace*{5mm}$

\section{Basic definitions and results}

Pseudo-hoops were introduced in \cite{Geo16} as a generalization of hoops which were originally defined and studied   by Bosbach in \cite{Bos1} and \cite{Bos2} under the name of complementary semigroups. 
It was proved that a pseudo-hoop has the pseudo-divisibility condition and it is a meet-semilattice, so a bounded R$\ell$-monoid can be viewed as a bounded pseudo-hoop together with the join-semilattice property. In other words, a pseudo-hoop is a meet-semilattice ordered residuated, integral and divisible monoid. 
In what follows we recall some basic notions and results regarding the pseudo-hoops. 
We prove new properties of pseudo-hoops and we give a characterization of simple pseudo-hoops. 

\begin{Def} \label{psh-10} $\rm($\cite{Geo16}$\rm)$ A \emph{pseudo-hoop} is an algebra $(A, \odot, \ra, \rs, 1)$ of 
the type $(2,2,2,0)$ such that, for all $x,y,z \in A$:\\
$(A_1)$ $x\odot 1=1\odot x=x;$\\
$(A_2)$ $x\ra x=x\rs x=1;$\\
$(A_3)$ $(x\odot y) \ra z=x\ra (y\ra z);$\\
$(A_4)$ $(x\odot y) \rs z=y\rs (x\rs z);$\\
$(A_5)$ $(x\ra y)\odot x=(y\ra x)\odot y=x\odot(x\rs y)=y\odot(y\rs x)$.
\end{Def}

In the sequel, we will agree that $\odot$ has higher priority than the operations $\ra$, $\rs$.\\
If the operation $\odot$ is commutative, or equivalently $\ra \:=\: \rs$, then the pseudo-hoop 
is said to be \emph{hoop}. Properties of hoops were studied in \cite{Bos1}, \cite{Bos2} and \cite{Blok1}. \\
On the pseudo-hoop $A$ we define $x\leq y$ iff $x\ra y=1$ (equivalent to $x\rs y=1$) and $\leq$ is a partial 
order on $A$. 
For any $n\in{\mathbb N}$, we define inductively: \\
$\hspace*{2cm}$ $x^0=1$, $x^{n+1}=x^n \odot x=x \odot x^n$, \\
$\hspace*{2cm}$ $x\ra^0y =y$, $x\ra^n y = x\ra (x\ra^{n-1} y)$, $n\ge 1$, \\ 
$\hspace*{2cm}$ $x\rs^0y =y$, $x\rs^n y = x\rs (x\rs^{n-1} y)$, $n\ge 1$. \\
If $A$ is a pseudo-hoop we denote: \\
$\hspace*{2cm}$ $\Id(A)=\{x\in A \mid x^2=x\}$, the set of all \emph{idempotent} elements of $A$. \\  
If $\Id(A)=A$, then $A$ is said to be \emph{idempotent}. \\ 
A pseudo-hoop $A$ is \emph{bounded} if there is an element $0\in A$ such that $0\leq x$ for all $x\in A$.\\
In the sequel we will also refer to the pseudo-hoop $(A, \odot, \ra, \rs,1)$ by its universe $A$. \\
Let $(A, \odot, \ra, \rs,0,1)$ be a bounded pseudo-hoop. We define two negations $^-$ and $^{\sim}$: 
for all $x \in A$, $\:\: x^-=x \ra 0$, $\:\:$ $x^{\sim}=x \rs 0$. 
If $A$ is a bounded pseudo-hoop we denote: \\ 
$\hspace*{2cm}$ $\Inv(A)=\{x\in A \mid x^{-\sim}=x^{\sim-}= x\}$, the set of all \emph{involutive} elements of $A$, \\
$\hspace*{2cm}$ $\Den(A)=\{x\in A \mid x^{-\sim}=x^{\sim-}= 1\}$, the set of all \emph{dense} elements of $A$, \\ 
If a bounded pseudo-hoop $A$ satisfies $x^{-\sim}=x^{\sim-}$ for all $x\in A$, then $A$ is called a \emph{good} pseudo-hoop. \\
Pseudo BCK-algebras were introduced by G. Georgescu and A. Iorgulescu in \cite{Geo15} as algebras 
with "two differences", a left- and right-difference, instead of one $*$ and with a constant element $0$ as 
the least element. 

\begin{Def} \label{psh-10-10} $\rm($\cite{Geo15}$\rm)$ A \emph{pseudo-BCK algebra} (more precisely, 
\emph{reversed left-pseudo-BCK algebra}) is a structure ${\mathcal A}=(A,\leq,\ra,\rs,1)$ where $\leq$ is a binary relation on $A$, $\ra$ and $\rs$ are binary operations on $A$ and $1$ is an element of $A$ satisfying, 
for all $x,y,z \in A$, the  axioms:\\
$(bck_1)$ $x \ra y \leq (y \ra z) \rs (x \ra z)$ and $x \rs y \leq (y \rs z) \ra (x \rs z)$; \\ 
$(bck_2)$ $x \leq (x \ra y) \rs y$ and $x \leq (x \rs y) \ra y$;\\
$(bck_3)$ $x \leq x$;\\
$(bck_4)$ $x \leq 1$;\\
$(bck_5)$ if $x \leq y$ and $y \leq x$, then $x = y$;\\
$(bck_6)$ $x \leq y$ iff $x \ra y = 1$ iff $x \rs y = 1$. 
\end{Def}

A pseudo-BCK algebra with \emph{(pP) condition} (i.e. with \emph{pseudo-product} condition) or  a \emph{pseudo-BCK(pP) algebra} for short, is a pseudo-BCK algebra 
${\mathcal A}=(A,\leq,\rightarrow,\rightsquigarrow,1)$ satisfying (pP) condition:\\
(pP) there exists, for all 
   $x,y \in A$, $x\odot y=min\{z \mid x \leq y \rightarrow z\}=min\{z \mid y \leq x \rightsquigarrow z\}$. \\
For more details about the properties of a pseudo-BCK algebra we refer te reader to \cite{Geo15}, \cite{Ior3}, and 
\cite{Kuhr1}. 
Commutative pseudo BCK-algebras were originally defined by G. Georgescu and A. Iorgulescu in \cite{Geo15} 
under the name of \emph{semilattice-ordered pseudo BCK-algebras}, while properties of these structures were 
investigated by J. K{\"u}hr in \cite{Kuhr1}, \cite{Kuhr2}. \\
It was proved in \cite{Ciu8} that any pseudo-hoop is a pseudo-BCK algebra with pseudo-product. 
It follows that all the properties of a pseudo-BCK algebra with pseudo-product proved in \cite{Ior1} and 
\cite{Ior13} are also valid in a pseudo-hoop. 

\begin{prop} \label{psh-20} $\rm($\cite{Geo16}, \cite{Ciu2}$\rm)$ In every pseudo-hoop $(A, \odot, \ra, \rs,1)$ the following hold for all $x, y, z\in A:$ \\
$(1)$ $(A, \odot,1)$ is a monoid;\\ 
$(2)$ $(A, \leq)$ is a meet-semillatice with $x\wedge y=(x\rightarrow y)\odot x=x\odot(x\rightsquigarrow y);$\\
$(3)$ $x\odot y\leq z$ iff $x\leq y \rightarrow z$ iff $y\leq x\rightsquigarrow z;$\\
$(4)$ $x\odot y\leq x\wedge y$, $\: x\leq y \rightarrow x$ and $x\leq y \rightsquigarrow x;$\\
$(5)$ $x \rightarrow y \leq (y \rightarrow z) \rightsquigarrow (x \rightarrow z)$ and 
      $x \rightsquigarrow y \leq (y \rightsquigarrow z) \rightarrow (x \rightsquigarrow z);$ \\ 
$(6)$ $x\leq y$ implies $z\rightarrow x \leq z\rightarrow y$ and 
                        $z\rightsquigarrow x \leq z\rightsquigarrow y;$\\
$(7)$ $x\leq y$ implies $y\rightarrow z \leq x\rightarrow z$ and $y\rightsquigarrow z \leq x\rightsquigarrow z;$ \\
$(8)$ $x\rightarrow y\wedge z=(x\rightarrow y)\wedge (x\rightarrow z)$ and 
      $x\rightsquigarrow y\wedge z=(x\rightsquigarrow y)\wedge (x\rightsquigarrow z);$ \\
$(9)$ $y\leq x\rightarrow y\odot x$ and $y\leq x\rightsquigarrow x\odot y;$ \\
$(10)$ $x\ra (y\rs z)=y\rs (x\ra z);$ \\
$(11)$ $x\ra y\le (z\ra x)\ra (z\ra y)$ and $x\rs y\le (z\rs x)\rs (z\rs y);$ \\
$(12)$ $x\ra y\le (x\odot z)\ra (y\odot z)$ and $x\rs y\le (z\odot x)\rs (z\odot y);$ \\
$(13)$ $x\leq y$ implies $x\odot z\le y\odot z$ and $z\odot x\le z\odot y;$ 
$(14)$ $x \leq (x \ra y) \rs y$ and $x \leq (x \rs y) \ra y$. 
\end{prop}

\begin{prop} \label{psh-20-05} $\rm($\cite[Prop. 3.1]{DvKo}$\rm)$ Let $(A, \odot, \ra, \rs,1)$ be a pseudo-hoop and 
$a\in \mbox{Id}(A)$. Then the following hold for all $x\in A$: \\
$(1)$ $a\odot x=a\wedge x=x\odot a;$ \\
$(2)$ $a\ra x=a\rs x$. 
\end{prop}

\begin{prop} \label{psh-20-05-10} In any pseudo-hoop $(A, \odot, \ra, \rs,1)$ the following hold 
for all $x, y\in A:$  \\   
$(1)$ $x\odot y=x\odot (x\rs x\odot y)=(y\ra x\odot y)\odot y;$ \\
$(2)$ $(x\ra y)\rs (y\ra x)=y\ra x$ and $(x\rs y)\ra (y\rs x)=y\rs x$. 
\end{prop}
\begin{proof}
Let $(A, \odot, \ra, \rs,1)$ be a pseudo-hoop and let $x, y\in A$. \\ 
$(1)$ According to Proposition \ref{psh-20}$(2)$ we have: \\
$\hspace*{2cm}$ $x\odot y=x\wedge (x\odot y)=(x\ra x\odot y)\odot x=x\odot (x\rs x\odot y)$, \\
$\hspace*{2cm}$ $x\odot y=y\wedge (x\odot y)=(y\ra x\odot y)\odot y=y\odot (y\rs x\odot y)$. \\
Hence $x\odot y=x\odot (x\rs x\odot y)=(y\ra x\odot y)\odot y$. \\
$(2)$ Denote $z=(x\ra y)\rs x$. Since $x\le z$, by Proposition \ref{psh-20}$(7)$,$(14)$ we get: \\ 
$\hspace*{2cm}$ $z\ra y\le x\ra y\le ((x\ra y)\rs x)\ra x=z\ra x$, \\ 
hence $(z\ra y)\odot z\le x$. 
By $(A_5)$, $(y\ra z)\odot y=(z\ra y)\odot z\le x$, thus $y\ra z\le y\ra x$. 
On the other hand, from $x\le z$, by Proposition \ref{psh-20}$(6)$ we have $y\ra x\le y\ra z$,  
that is $y\ra z=y\ra x$. 
Applying Proposition \ref{psh-20}$(10)$ we get: \\
$\hspace*{2cm}$ $(x\ra y)\rs (y\ra x)=y\ra ((x\ra y)\rs x)=y\ra z=y\ra x$. \\
%Since $y\ra x\le (x\ra y)\rs (y\ra x)$, it follows that $(x\ra y)\rs (y\ra x)=y\ra x$. 
Similarly $(x\rs y)\ra (y\rs x)=y\rs x$. 
\end{proof}

\begin{prop} \label{psh-20-10} $\rm($\cite{Geo16}, \cite{Ciu2}$\rm)$ In every bounded pseudo-hoop 
$(A, \odot, \ra, \rs,1)$ the following hold for all $x, y\in A$:\\       
$(1)$ $x \leq x^{-\sim}$ and $x \leq x^{\sim-};$\\
$(2)$ $x \rightarrow y^{\sim} = y \rightsquigarrow x^-$ and 
      $x \rightsquigarrow y^- = y \rightarrow x^{\sim};$\\
$(3)$ $x^{-\sim-}=x^-$ and $x^{\sim-\sim}=x^{\sim};$ \\
$(4)$ $x \rightarrow y^{-\sim}=y^- \rightsquigarrow x^- = x^{-\sim} \rightarrow y^{-\sim}$ and 
       $x \rightsquigarrow y^{\sim-}=y^{\sim} \rightarrow x^{\sim} = x^{\sim-} \rightsquigarrow y^{\sim-};$\\
$(5)$ $x\rightarrow y^{-}=(x\odot y)^{-}$ and $x\rightsquigarrow y^{\sim}=(y\odot x)^{\sim};$ \\
$(6)$ $x\ra y\le y^-\rs x^-$ and $x\rs y\le y^{\sim}\ra x^{\sim}$.                 
\end{prop}

\begin{prop} \label{psh-30} $\rm($\cite[Prop. 2.4]{Ciu2}$\rm)$ If $(A, \odot, \ra, \rs,0,1)$ 
is a good pseudo-hoop, then the following hold for all $x, y\in A$: \\
$(1)$ $(x^{-\sim}\rightarrow x)^{\sim}=(x^{-\sim}\rightsquigarrow x)^-=0;$ \\
$(2)$ $(x \rightarrow y)^{-\sim}=x^{-\sim} \rightarrow y^{-\sim} \:$ and
      $\: (x \rightsquigarrow y)^{-\sim}=x^{-\sim} \rightsquigarrow y^{-\sim};$  \\
$(3)$ $(x\wedge y)^{-\sim}=x^{-\sim}\wedge y^{-\sim};$ \\ 
$(4)$ $x\ra y^{-}=x^{-\sim}\ra y^{-}$ and $x\rs y^{\sim}=x^{-\sim}\rs y^{\sim}$. 
\end{prop}

\begin{cor} \label{psh-30-10} Any good pseudo-hoop $(A, \odot, \ra, \rs,0,1)$ satisfies the following identities 
for all $x, y\in A$: \\
$\hspace*{3cm}$ $(x\ra y)^{-\sim}=x\ra y^{-\sim}$, \\
$\hspace*{3cm}$ $(x\rs y)^{-\sim}=x\rs y^{-\sim}$.
\end{cor}
\begin{proof} It follows from Propositions \ref{psh-30}$(2)$ and \ref{psh-20-10}$(4)$.
\end{proof}

\begin{rem} \label{psh-30-20} Due to Corollary \ref{psh-30-10}, we say that any good pseudo-hoop 
$(A, \odot, \ra, \rs,0,1)$ has the \emph{Glivenko property}.
\end{rem}

\begin{lemma} \label{psh-40} If $(A, \odot, \ra, \rs,0,1)$ is a good pseudo-hoop, then \\
$\hspace*{3cm}$ $(x\odot y)^{-\sim}\ge x^{-\sim}\odot y^{-\sim}$, \\ 
for all $x, y\in A$.
\end{lemma} 
\begin{proof} Let $x, y\in A$. Applying Propositions \ref{psh-30}$(4)$ and \ref{psh-20-10}$(5)$,$(3)$ we have: \\
$\hspace*{2cm}$ $(x^{-\sim}\odot y^{-\sim})^{-}=x^{-\sim}\ra y^{-\sim-}=x\ra y^{-\sim-}=x\ra y^{-}= (x\odot y)^{-}$, \\ hence by Proposition \ref{psh-20-10}$(1)$ we get: \\
$\hspace*{2cm}$ $(x\odot y)^{-\sim}=(x^{-\sim}\odot y^{-\sim})^{-\sim}\ge x^{-\sim}\odot y^{-\sim}$. 
\end{proof}

If $A$ is a bounded pseudo-hoop, then the \emph{order} of $x\in A$, denoted $ord(x)$ is the smallest $n\in{\mathbb N}$ such that $x^n=0$. 
If there is no such $n$, then $ord(x)=\infty$.\\
We say that $A$ is \emph{locally finite} if for any $x\in A$, $x\neq 1$ implies $ord(x)<\infty$. \\
Let $(A,\odot, \rightarrow, \rightsquigarrow,1)$ be a pseudo-hoop. A non-empty subset $F$ of $A$ is a \emph{filter} of $A$ if for all $x,y\in A$ the following conditions are satisfied:\\
$(F_1)$ $x,y \in F$ implies $x\odot y \in F;$\\
$(F_2)$ $x\in F$ and $x\leq y$ implies $y\in F$.\\
A filter $F$ of $A$ is \emph{proper} if $F\neq A$.\\
A filter $H$ of $A$ is called \emph{normal} if for every $x,y\in A$, $x\rightarrow y\in A$ iff 
$x\rightsquigarrow y\in A$. \\
A \emph{maximal} filter or \emph{ultrafilter} is a proper filter $F$ of $A$  that is not included in any other proper filter of $A$. 
Denote by: \\
$\hspace*{2cm}$ ${\mathcal F}(A)$ the set of all filters of $A$, \\
$\hspace*{2cm}$ ${\mathcal F}_n(A)$ the set of all normal filters of $A$, \\
$\hspace*{2cm}$ ${\mathcal F}_m(A)$ the set of all maximal filters of $A$. \\ 
Obviously $\{\{1\},A\}\subseteq {\mathcal F}_n(A)\subseteq {\mathcal F}(A)$. \\
Given $H\in {\mathcal F}_n(A)$, the relation $\Theta_H$ on $A$ defined by $(x,y)\in \Theta_H$ iff 
$x\ra y\in H$ and $y\ra x\in H$ is a congruence on $A$. 
Then $H=[1]_{\Theta_H}$ and $A/H=(A/\Theta_H,\rightarrow, \rightsquigarrow,[1]_{\Theta_H})$ is a pseudo-hoop  
and we write $x/H=[x]_{\Theta_H}$ for every $x\in A$ (see \cite{Geo16}). \\
The function $\pi_H: A \longrightarrow A/H$ defined by $\pi_H(x)=x/H$ for any $x\in A$ is a surjective homomorphism which is called the \emph{canonical projection} from $A$ to $A/H$. One can easily prove that $\Ker(\pi_H)=H$. \\
%A pseudo-hoop $A$ is called \emph{simple} if $\{1\}$ is the unique proper normal filter of $A$.\\
%A pseudo-hoop $A$ is called \emph{strongly simple} if $\{1\}$ is the unique proper filter of $A$.\\
A pseudo-hoop $A$ is called \emph{simple} if $\{1\}$ is the unique proper filter of $A$.\\
%Obviously any strongly simple pseudo-hoop is simple. \\ 
%When $A$ is a hoop, since filters and normal filters coincide, the notions of simple and strongly simple 
%hoop coincide. \\
The subset $F \subseteq A$ is called a \emph{deductive system} of $A$ if it satisfies the following conditions:\\
$(DS_1)$ $1 \in F;$ \\
$(DS_2)$ for all $x, y \in A$, if $x, x \ra y \in F$, then $y \in F$. \\
Let $A$ be a pseudo-hoop. Then $F\subseteq A$ with $1\in F$ is a deductive
system of $A$ if and only if it satisfies the condition: \\
$(DS_2^{\prime})$ for all $x, y \in A$, if $x, x \rs y \in F$, then $y \in F$. \\
Let $A$ be pseudo-hoop and $F$ a nonempty subset of $A$.
Then the following are equivalent: \\
$(a)$ $F$ is a deductive system of $A;$ \\
$(b)$ $F$ is a filter of $A$. \\
If $X\subseteq A$, we denote by $<X>$ the filter generated by $X$. If $X=\{x\}$, then we use the notation 
$<x>$ instead of $<\{x\}>$, and $<x>$ is called the \emph{principal filter} generated by $x$. 

\begin{prop} \label{psh-50} $\rm($\cite{Geo16}$\rm)$ If $A$ is a pseudo-hoop and $X\subseteq A$, then \\ 
$<X>=\{y\in A\mid y\geq x_1\odot x_2\odot \dots\odot x_n$ for some $n\geq 1$ and 
      $x_1,x_2,\dots,x_n\in X\}$ \\
$\hspace*{1.0cm}$ $=\{y\in A\mid x_1\rightarrow(x_2\rightarrow(\dots(x_n\rightarrow y)\dots))=1$ 
                   for some $n\geq 1$ and $x_1,x_2,\dots,x_n\in X\}$ \\
$\hspace*{1.0cm}$ $=\{y\in A\mid x_1\rightsquigarrow(x_2\rightsquigarrow(\dots(x_n\rightsquigarrow y)\dots))=1$ 
                     for some $n\geq 1$ and $x_1,x_2,\dots,x_n\in X\}$. \\
In particular, the principal filter generated by an element $x\in A$ is \\
$\hspace*{2.0cm}$ $<x>=\{y\in A \mid x^n\le y\}$ for some $n\geq 1$ \\
$\hspace*{3.1cm}$ $=\{y\in A \mid x\ra^n y\}$ for some $n\geq 1$ \\
$\hspace*{3.1cm}$ $=\{y\in A \mid x\rs^n y\}$ for some $n\geq 1$.      
\end{prop}

\begin{prop} \label{psh-60} $\rm($\cite{Geo16}$\rm)$ For any pseudo-hoop $A$ the following are equivalent:\\
$(a)$ $A$ is simple;\\
$(b)$ for all $x\in A$, if $x\neq 1$ then $<x>=A$.
\end{prop}

\begin{prop} \label{psh-waj-110} Let $A$ be a pseudo-hoop. The following hold:\\
$(1)$ $A$ is simple if and only if for all $x, y\in A$, $x\ne 1$, there exists $n\in {\mathbb N}$ such that 
      $x\ra ^n y=1;$ \\
$(2)$ $A$ is simple if and only if for all $x, y\in A$, $x\ne 1$, there exists $n\in {\mathbb N}$ such that 
      $x\rs ^n y=1;$ \\      
$(3)$ if $A$ is simple, then for all $x, y\in A$, $y\ra x=x$ implies $x=1$ or $y=1;$ \\
$(4)$ if $A$ is simple, then for all $x, y\in A$, $y\rs x=x$ implies $x=1$ or $y=1$.   
\end{prop}
\begin{proof} 
$(1)$ If $A$ is simple, then for all $x\in A$, $x\ne 1$ we have $A=<x>=\{z\in A \mid x\ra^n z=1\}$ for some 
$n\in {\mathbb N}$. Since $y\in A$ and $A=<x>$, then there exists $n\in {\mathbb N}$ such that $x\ra ^n y=1$. 
Conversely, suppose that for all $x, y\in A$, $x\ne 1$ we have $x\ra ^n y=1$ for some $n\in {\mathbb N}$. \\
It follows that $y\in <x>$ for all $y\in A$, that is $<x>=A$, so $A$ is a simple pseudo-hoop. \\
$(2)$ Similarly as $(1)$. \\
$(3)$ Assume that $A$ is simple and let $x, y\in A$ such that $y\ra x=x$. It follows that $y\ra^n x=x$ 
for all $n\in {\mathbb N}$. If $y\ne 1$, then according to $(1)$, there exists $n_0\in {\mathbb N}$ such that 
$y\ra^{n_0} x=1$, hence $x=1$. \\
$(4)$ Similarly as $(3)$.
\end{proof}

A pseudo-hoop $(A,\odot, \rightarrow, \rightsquigarrow,1)$ is said to be \emph{cancellative} if the monoid $(A,\odot,1)$ is cancellative, that is $x\odot a = y\odot a$ implies $x=y$ and $a\odot x = a\odot y$ implies $x=y$ for all $x,y,a \in A$. 

\begin{prop} \label{psh-60-10} $\rm($\cite{Geo16}$\rm)$
A pseudo-hoop $A$ is cancellative iff $y\ra x\odot y=x$ and $y\rs y\odot x=x$ for all $x,y\in A$. 
\end{prop}

\begin{prop} \label{psh-70} $\rm($\cite{Geo16}$\rm)$ Let $A$ be a cancellative pseudo-hoop. Then for all $x,y,z \in A$ the following hold:\\
$(1)$ $x\ra y=x\odot z \ra y\odot z$ and $x\rs y=z\odot x \rs z\odot y;$\\
$(2)$ $x\leq y$ iff $x\odot z \leq y\odot z$ iff $z\odot x \leq z\odot y$.
\end{prop}

\begin{ex} \label{psh-110} $\rm($\cite{Geo16}$\rm)$ Let $\textbf{G}=(G,+,-,0,\vee,\wedge)$ be an arbitrary $\ell$-group and $N(G)$ the negative cone of $\textbf{G}$, that is $N(G)=\{x\in G \mid x\leq 0\}$. On $N(G)$ we define the 
following operations:\\
$\hspace*{3cm}$ $x\odot y=x+y$, \\
$\hspace*{3cm}$ $x\ra y=(y-x)\wedge 0$, \\
$\hspace*{3cm}$ $x\rs y=(-x+y)\wedge 0$. \\
Then $\textbf{N(G)}=(N(G),\odot,\ra, \rs, 0)$ is a cancellative pseudo-hoop.
\end{ex}

Let $A$ be a pseudo-hoop. In the next sections we will also use the notations:\\
$\hspace*{3cm}$ $x\vee_1 y= (x\ra y)\rs y$ and $x\vee_2 y= (x\rs y)\ra y$, \\
for all $x, y\in A$. 
If $A$ is bounded, then obviously $x^{-\sim}= x\vee_1 0$ and $x^{\sim-}= x\vee_2 0$.                 

\begin{prop} \label{psh-80} $\rm($\cite{Ciu10}$\rm)$ In any pseudo-hoop $A$ the following hold for all $x,y \in A$:\\
$(1)$ $1 \vee_1 x = x \vee_1 1 = 1 = 1 \vee_2 x = x \vee_2 1;$ \\
$(2)$ $x \leq y$ implies $x \vee_1 y = y$ and $x \vee_2 y = y;$\\
$(3)$ $x \vee_1 x = x \vee_2 x = x;$\\
$(4)$ if $x_1 \le x_2$ and $y_1 \le y_2,$ then $x_1\vee_1 y_1 \le x_2\vee_1 y_2$ and 
         $x_1\vee_2 y_1 \le x_2\vee_2 y_2;$ \\
$(5)$ $x, y\le x\vee_1 y, x\vee_2 y$.
\end{prop}

\begin{prop} \label{psh-90} $\rm($\cite{Ciu10}$\rm)$ Let $A$ be a pseudo-hoop. Then for all $x,y \in A$ the following hold:\\
$(1)$ $x \vee_1 y \ra y = x \ra y$ and $x \vee_2 y \rs y = x \rs y;$ \\
$(2)$ $x \vee_1 y \ra x = y \ra x$ and $x \vee_2 y \rs x = y \rs x$.      
\end{prop}

A \emph{bounded non-commutative R$\ell$-monoid} is an algebra $(A,\odot,\vee,\wedge,\ra,\rs,0,1)$ of the type $(2,2,2,2,2,0,0)$ satisfying the following conditions:\\
$(R\ell_1)$ $(A,\odot,1)$ is a monoid;\\
$(R\ell_2)$ $(A,\vee,\wedge,0,1)$ is a bounded lattice with bounds $0$ and $1$ (bottom and top);\\
$(R\ell_3)$ $x\odot y\leq z$ iff $x\leq y\ra z$ iff $y\leq x\rs z$ for all $x,y,z \in A;$\\
$(R\ell_4)$ $(x\ra y)\odot x=y\odot (y\rs x)=x\wedge y$ for all $x,y \in A$. \\
For more details about the properties of a bounded R$\ell$-monoid we refer the reader to \cite{DvRa1} and 
\cite{DvRa2}. \\
A bounded non-commutative R$\ell$-monoid satisfying the \emph{pre-linearity} condition: \\
(prel) $(x\ra y)\vee (y\ra x)=(x\rs y)\vee (y\rs x)=1$, \\
is a \emph{pseudo-BL algebra}. \\
If the algebra $(A,\odot,\vee,\wedge,\ra,\rs,0,1)$ satisfies conditions $(R\ell_1)$, $(R\ell_2)$, $(R\ell_3)$ and 
(prel), then it is a \emph{pseudo-MTL algebra}. 

Let $A, B$ be two pseudo-hoops. A map $f: A\longrightarrow B$ is called a \emph{pseudo-hoop homomorphism} if 
it satisfies the following axioms for all $x, y\in A$: \\
$(i)$ $f(x\odot y)=f(x)\odot f(y);$ \\
$(ii)$ $f(x\ra y)=f(x)\ra f(y);$ \\
$(iii)$ $f(x\rs y)=f(x)\rs f(y)$. \\ 
If $A, B$ are bounded pseudo-hoops, then $f: A\longrightarrow B$ is a \emph{bounded pseudo-hoop homomorphism} if 
it satisfies axioms $(i)-(iii)$ and the following axiom:\\
$(iv)$ $f(0)=0$. \\
If $B=A$, then $f$ is called a \emph{pseudo-hoop endomorphism}. \\
One can easily check that, if $f$ is a pseudo-hoop homomorphism, then: \\
$(1)$ $f(1)=1;$ \\
$(2)$ $f(x\wedge y)=f(x)\wedge f(y);$ \\
$(3)$ $x\le y$ implies $f(x)\le f(y)$. \\
If $f$ is a bounded pseudo-hoop homomorphism, then the following hold: \\ 
$(4)$ $f(x^{-})=f(x)^{-};$ \\ 
$(5)$ $f(x^{\sim})=f(x)^{\sim}$. \\
(We use the same notations for the operations in both pseudo-hoops\textsl{}, but the reader must be aware that they are different). \\
Denote $\Ker(f)=\{x\in A \mid f(x)=1\}$. \\

The Bosbach states and state-morphisms on bounded pseudo-hoops were defined and studied in \cite{Ciu10}. 

\begin{Def} \label{psh-s-10}
A \emph{Bosbach state} on a bounded pseudo-hoop $(A,\ra,\rs,0,1)$ is a function $s: A\longrightarrow [0, 1]$ 
such that the following axioms hold for all $x, y\in A:$ \\
$(bs_1)$ $s(1)=0$ and $s(1)=1;$ \\
$(bs_2)$ $s(x)+s(x\ra y)=s(y)+s(y\ra x);$ \\
$(bs_3)$ $s(x)+s(x\rs y)=s(y)+s(y\rs x)$. \\
Denote by $\mathcal{BS}(A)$ the set of all Bosbach states on the bounded pseudo-hoop $A$. 
\end{Def} 

\begin{prop} \label{psh-s-20} Let $A$ be a bounded pseudo-hoop and let $s\in \mathcal{BS}(A)$. 
Then the following hold for all $x, y\in A$: \\
$(1)$ $s(x^{-})=s(x^{\sim})=1-s(x);$ \\
$(2)$ $s(x^{-\sim})=s(x^{\sim-})=s(x)$.  
\end{prop}

Let $s\in \mathcal{BS}(A)$ and define $\Ker(s)=\{x\in A \mid s(x)=1\}$, called the \emph{kernel} of $s$. \\
One can easily check that $\Ker(s)\in {\mathcal F}(A)$. \\  
The measures on bounded pseudo-BCK algebras were defined and studied in \cite{Ciu6} and these results are also 
valid for bounded pseudo-hoops. 

\begin{Def} \label{psh-m-10}
Let $(A,\ra,\rs,0,1)$ be a bounded pseudo-hoop. A mapping $m:A\longrightarrow [0, \infty)$ such that $m(0)=1$ and 
$m(x\ra y)=m(x\rs y)=m(y)-m(x)$ whenever $y\leq x$ is said to be a \emph{state-measure}. 
Denote by $\mathcal{M}(A)$ the set of all state-measures on $A$.
\end{Def}

\begin{prop} \label{psh-m-20} If $m\in \mathcal{M}(A)$, then the following hold for all $x, y\in A$: \\
$(1)$ $m(1)=0;$ \\
$(2)$ $m(x)\geq m(y)$ whenever $x\leq y;$  \\
$(3)$ $m(x\vee_1 y) = m(y\vee_1 x)$ and $m(x\vee_2 y) = m(y\vee_2 x);$ \\
$(4)$ $m(x\vee_1 y) = m(x\vee_2 y);$ \\
$(5)$ $m(x\ra y)=m(x\rs y);$ \\
$(6)$ $m(x^{-})=m(x^{\sim})=1-m(x);$ \\
$(7)$ $m(x^{-\sim})=m(x^{\sim-})=m(x)$.
\end{prop}

If $m\in \mathcal{M}(A)$, then $\Ker_0(m)=\{x\in A \mid m(x)=0\} \in {\mathcal F}_n(A)$ ($\Ker_0(m)$ is called 
the \emph{kernel} of $m$). \\

$\vspace*{5mm}$

\section{Wajsberg and involutive pseudo-hoops}

In this section we give a characterization of bounded Wajsberg pseudo-hoops and we recall some properties 
of these structures. We define the notion of involutive pseudo-hoop, we give a characterization of involutive pseudo-hoops and we investigate their properties. 
We introduce the notion of a normal pseudo-hoop and we prove that the set of all involutive elements of a 
normal pseudo-hoop $A$ is a subalgebra of $A$. 
We define the notion of an Archimedean pseudo-hoop and we prove that a pseudo-hoop is Archimedean if and only 
if it is a linearly ordered Wajsberg pseudo-hoop. 
As a consequence, any simple pseudo-hoop is a linearly ordered Wajsberg pseudo-hoop.

\begin{Def} \label{psh-inv-10} A pseudo-hoop $(A,\odot, \ra, \rs,1)$ is said to be \emph{Wajsberg} if it 
satisfies the following conditions:\\
$(W_1)$ $x\vee_1 y=y\vee_1 x;$\\
$(W_2)$ $x\vee_2 y=y\vee_2 x$. 
\end{Def}

\begin{Def} \label{psh-inv-10-10} A pseudo-hoop $(A,\odot, \ra, \rs,1)$ is said to be \emph{basic} if it 
satisfies the following conditions:\\
$(B_1)$ $(x\ra y)\ra z \leq ((y\ra x)\ra z)\ra z;$\\
$(B_2)$ $(x\rs y)\rs z \leq ((y\rs x)\rs z)\rs z$. 
\end{Def}

We recall that every %strongly 
simple basic pseudo-hoop is a linearly ordered Wajsberg pseudo-hoop 
(\cite[Cor. 4.15]{Geo16}) and every bounded Wajsberg pseudo-hoop is a bounded non-commutative R$\ell$-monoid 
$\rm($\cite[Prop. 2.14]{Ciu10}$\rm)$. \\
If $(A,\odot, \ra, \rs,1)$ is a Wajsberg pseudo-hoop, then $(A,\le)$ is a distributive lattice with: \\
$\hspace*{3cm}$ $x\vee y=x\vee_1 y=x\vee_2 y$, \\
$\hspace*{3cm}$ $x\wedge y=(x\ra y)\odot x=x\odot(x\rs y)$. 

\begin{theo} \label{psh-waj-10} Let $(A,\odot, \ra, \rs, 0, 1)$ be a pseudo-hoop.  
The following are equivalent for all $x, y\in A$: \\
$(a)$ $A$ is a Wajsberg pseudo-hoop; \\
$(b)$ $x\ra y=y\vee_1 x\ra y$ and $x\rs y=y\vee_2 x\rs y;$ \\
$(c)$ $x\vee_1 y=(x\vee_1 y)\vee_1 x$ and $x\vee_2 y =(x\vee_2 y)\vee_2 x;$ \\
$(d)$ $x\le y$ implies $y=y\vee_1 x=y\vee_2 x$.
\end{theo}
\begin{proof} Similarly as \cite[Th. 3.9]{Ciu21}. 
\end{proof}

\begin{prop} \label{psh-waj-10-10} Any finite Wajsberg pseudo-hoop is a Wajsberg hoop. 
\end{prop}
\begin{proof}
According to \cite[Cor. 3.6]{Kuhr2} any finite pseudo-BCK algebra is a BCK-algebra. 
Since a Wajsberg pseudo-hoop is a commutative pseudo-BCK algebra, it follows that it is a Wajsberg hoop.    
\end{proof}

According to \cite{Ior1}, a bounded commutative pseudo-BCK algebra $A$ is a Wajsberg pseudo-hoop where 
$x\odot y=(x\ra y^{-})^{\sim}=(y\rs x^{\sim})^{-}$ and 
$x\wedge y=(x^{-}\vee y^{-})^{\sim}=(x^{\sim}\vee y^{\sim})^{-}$. 
Hence the bounded Wajsberg pseudo-hoops are term equivalent to bounded commutative pseudo-BCK algebras. 
Based on this result, we can transfer properties of bounded commutative pseudo-BCK algebras to bounded 
Wajsberg pseudo-hoops. 

\begin{ex} \label{psh-waj-10-20} $\rm($\cite{Geo16}$\rm)$ Let $\textbf{G}=(G,\vee,\wedge,+,-,0)$ be an arbitrary $\ell$-group. For an arbitrary element $u\in G$, $u\geq 0$ define on the set $G[u]=[0,u]$ the operations:\\
$\hspace*{3cm}$ $x\odot y:=(x-u+y)\vee 0$,\\
$\hspace*{3cm}$ $x\ra y:=(y-x+u)\wedge u$,\\
$\hspace*{3cm}$ $x\rs y:=(u-x+y)\wedge u$.\\
Then $\textbf{G[u]}=(G[u],\odot,\ra, \rs, 0, u)$ is a bounded Wajsberg pseudo-hoop. 
\end{ex}

\begin{exs} \label{psh-waj-10-20-10} $\rm($\cite{Jip1}$\rm)$
$(1)$ Representable Brouwerian algebras are idempotent basic hoops. \\
$(2)$ Generalized Boolean algebras are idempotent Wajsberg hoops.
\end{exs}

\begin{prop} \label{psh-waj-10-60} If $A$ is a Wajsberg pseudo-hoop, then: \\
$\hspace*{2cm}$ $x\vee y\ra z=(x\ra z)\wedge (y\ra z)$ and $x\vee y\rs z=(x\rs z)\wedge (y\rs z)$, \\
for all $x, y, z\in A$. 
\end{prop}
\begin{proof}
Since $x, y\le x\vee y$, we have $x\vee y\ra z\le x\ra z$ and $x\vee y\ra z\le y\ra z$, so 
$x\vee y\ra z\le (x\ra z)\wedge (y\ra z)$. 
If $u\le (x\ra z)\wedge (y\ra z)$, then $u\le x\ra z$ and $u\le y\ra z$, so $u\odot x\le z$ and $u\odot y\le z$. 
It follows that $x, y\le u\rs z$, hence $x\vee y\le u\rs z$, that is $u\odot (x\vee y)\le z$, 
and $u\le x\vee y\ra z$. 
Thus $(x\ra z)\wedge (y\ra z)\le x\vee y\ra z$, and we conclude that $x\vee y\ra z=(x\ra z)\wedge (y\ra z)$. 
Similarly $x\vee y\rs z=(x\rs z)\wedge (y\rs z)$.
\end{proof}

\begin{Def} \label{psh-inv-10-20} A bounded pseudo-hoop is said to be an \emph{involutive} pseudo-hoop if 
$\Inv(A)=A$.
\end{Def} 

Obviously, if $A$ is involutive, then $A$ is good and $\Den(A)=\{1\}$. \\  
Taking $y=0$ in $(W_1)$ and $(W_2)$, it follows that a bounded Wajsberg pseudo-hoop is involutive.
As a consequence, every bounded Wajsberg pseudo-hoop is good. 

\begin{ex} \label{psh-waj-10-30} Let $(A,\ra,\rs,0,1)$ be a bounded pseudo-hoop and $m\in \mathcal {M}(A)$.
Then, by \cite[Prop. 4.3, Th. 4.8]{Ciu6}, $\Ker_0(m)=\{x\in A\mid m(x)=0\}\in {\mathcal F}_{n}(A)$ and 
$A/\Ker_0(m)$ is a bounded Wajsberg pseudo-hoop, that is an involutive pseudo-hoop. 
\end{ex}

\begin{ex} \label{psh-waj-10-40} Let $(A,\ra,\rs,0,1)$ be a bounded pseudo-hoop.  
A system ${\mathcal S}$ of state-measures on $A$ is an \emph{order-determing system} on $A$ if for all 
state-measures $m\in {\mathcal S}$, $m(x)\ge m(y)$ implies $x\le y$. 
If $A$ possesses an order-determing system ${\mathcal S}$ of measures then $A$ is a bounded Wajsberg pseudo-hoop. 
Indeed suppose that for $x, y\in A$ we have $x\le y$. Then, by \cite[Prop. 4.3]{Ciu6}, 
$m(y\vee_1 x)=m(y\vee_2 x)=m(y)$, for all $m\in {\mathcal S}$. 
Since ${\mathcal S}$ is order-determing then $y\vee_1 x=y\vee_2 x=y$. 
According to Theorem \ref{psh-waj-10}, $A$ is a bounded Wajsberg pseudo-hoop, thus it is an involutive pseudo-hoop. 
\end{ex}

\begin{prop} \label{psh-waj-20} Let $A$ be an involutive pseudo-hoop. Then the following hold for all $x,y \in A$: \\
$(1)$ $x \leq y$ iff $y^{-} \leq x^-$ iff $y^{\sim} \leq x^{\sim}$;\\
$(2)$ $x \ra y = y^{-} \rs x^{-}$ and $x \rs y = y^{\sim} \ra x^{\sim}$;\\
$(3)$ $x^{\sim} \ra y = y^{-} \rs x$ and $x\rs y^{-}=y\ra x^{\sim};$ \\
$(4)$ $(x \ra y^-)^{\sim}=(y \rs x^{\sim})^{-}$. 
\end{prop}
\begin{proof}
Similarly as \cite[Prop. 3.1]{Ciu115}. 
\end{proof}

\begin{prop} \label{psh-waj-20-05} Let $A$ be an involutive pseudo-hoop.  
Then the following hold for all $x,y \in A$: \\
$(1)$ $(x^{-}\vee y^{-})^{\sim}=(x^{\sim}\vee y^{\sim})^{-}=x\wedge y;$ \\
$(2)$ $(x\wedge y)^{-}=x^{-}\vee y^{-}$ and $(x\wedge y)^{\sim}=x^{\sim}\vee y^{\sim};$ \\
$(3)$ $(x^{-}\wedge y^{-})^{\sim}=(x^{\sim}\wedge y^{\sim})^{-}=x\vee y$.
\end{prop}
\begin{proof}
Similarly as \cite[Prop. 3.2, Prop. 3.3, Cor. 3.1]{Ciu115}. 
\end{proof}

\begin{theo} \label{psh-waj-30} $\rm($\cite[Th. 3.1]{Ciu115}$\rm)$ Every bounded locally finite pseudo-hoop 
is an involutive pseudo-hoop. 
\end{theo}

\begin{theo} \label{psh-waj-40} Let $(A,\odot, \ra, \rs, 0, 1)$ be a bounded pseudo-hoop.  
The following are equivalent: \\
$(a)$ $A$ is an involutive pseudo-hoop;\\
$(b)$ $x \ra y = y^{-} \rs x^{-}$ and $x\rs y = y^{\sim} \ra x^{\sim};$ \\
$(c)$ $x^{\sim} \ra y = y^{-} \rs x$ and $x^{-} \rs y = y^{\sim} \ra x;$ \\
$(d)$ $x^{-} \leq y$ implies $y^{\sim} \leq x$ and $x^{\sim} \leq y$ implies $y^{-} \leq x$.
\end{theo}
\begin{proof} Similarly as \cite[Th. 3.2]{Ciu115}. 
\end{proof}

\begin{Def} \label{psh-waj-70} A bounded pseudo-hoop $A$ is called \emph{normal} if 
$(x\odot y)^{-\sim}=x^{-\sim}\odot y^{-\sim}$ and $(x\odot y)^{\sim-}=x^{\sim-}\odot y^{\sim-}$, for all $x, y\in A$. 
\end{Def} 

\begin{exs} \label{psh-waj-80} $(1)$ Any bounded idempotent pseudo-hoop is normal. \\
$(2)$ Every good pseudo-BL algebra is normal (\cite[Prop. 2.2]{Rac5}).
\end{exs}

\begin{prop} \label{psh-waj-100} If $A$ is a normal Wajsberg pseudo-hoop, then $\Inv(A)$ is a subalgebra of $A$. 
\end{prop}
\begin{proof} 
Let $A$ be a normal Wajsberg pseudo-hoop and $x, y\in \Inv(A)$. \\
By Proposition \ref{psh-30}, $x\wedge y, x\ra y, x\rs y\in \Inv(A)$ and \\
$\hspace*{2cm}$
$(x\vee y)^{-\sim}=((x\ra y)\rs y)^{-\sim}=(x\ra y)^{-\sim}\rs y^{-\sim}$ \\
$\hspace*{3.8cm}$
$=(x^{-\sim}\ra y^{-\sim})\rs y^{-\sim}=(x\ra y)\rs y=x\vee y$, \\
hence $x\vee y\in \Inv(A)$. 
Since $A$ is normal, $(x\odot y)^{-\sim}=x^{-\sim}\odot y^{-\sim}=x\odot y$, thus $x\odot y\in \Inv(A)$. 
It follows that $\Inv(A)$ is a subalgebra of $A$. 
\end{proof}

\begin{Def} \label{psh-waj-100-10} A pseudo-hoop $(A,\odot, \ra, \rs,1)$ is said to be \emph{Archimedean} if it 
satisfies $(A)$ condition for all $x, y\in A$: \\
$(A)$ $\hspace*{0.5cm}$ $(y\ra x=x$ implies $x=1$ and $y=1)$ or $(y\rs x=x$ implies $x=1$ or $y=1)$. \\
\end{Def}

\begin{rem} \label{psh-waj-100-20} 
$(1)$ According to Proposition \ref{psh-waj-110}, any simple pseudo-hoop is Archimedean. \\
$(2)$ It was proved in \cite{Geo16} that a basic Archimedean pseudo-hoop is a linearly ordered Wajsberg pseudo-hoop. 
\end{rem}

In what follows we extend to pseudo-hoops a result proved in \cite{Blok1} for the case of hoops. 

\begin{theo} \label{psh-waj-120} A pseudo-hoop is Archimedean if and only if it is a linearly ordered Wajsberg pseudo-hoop. 
\end{theo}
\begin{proof} 
Let $A$ be a pseudo-hoop such that for all $x, y\in A$, $y\ra x=x$ implies $x=1$ or $y=1$. \\
According to Proposition \ref{psh-20-05-10}$(2)$, for all $x, y\in A$ we have $(x\ra y)\rs (y\ra x)=y\ra x$.  
Hence by hypothesis, $x\ra y=1$ or $y\ra x=1$. It follows that $x\le y$ or $y\le x$, thus $A$ is linearly ordered. 
Let $x, y\in A$ and assume without loss of generality that $x<y$.  
Applying Proposition \ref{psh-90}, $(A_3)$ and $(A_5)$ we get: \\ 
$\hspace*{1cm}$ $(y\vee_1 x\ra y)\ra (y\ra x)=(y\vee_1 x\ra y)\ra (y\vee_1 x\ra x)$ \\
$\hspace*{5.1cm}$ $=(y\vee_1 x\ra y)\odot (y\vee_1 x)\ra x$ \\
$\hspace*{5.1cm}$ $=(y\ra y\vee_1 x)\odot y\ra x$ \\
$\hspace*{5.1cm}$ $=(y\ra y\vee_1 x)\ra (y\ra x)$ \\
$\hspace*{5.1cm}$ $=1\ra (y\ra x)=y\ra x$. \\
Similarly, by Proposition \ref{psh-90}, $(A_4)$ and $(A_5)$ we have: \\ 
$\hspace*{1cm}$ $(y\vee_2 x\rs y)\rs (y\rs x)=(y\vee_2 x\rs y)\rs (y\vee_2 x\rs x)$ \\
$\hspace*{5.1cm}$ $=(y\vee_2 x)\odot (y\vee_2 x\rs y)\rs x$ \\
$\hspace*{5.1cm}$ $=y\odot (y\rs y\vee_2 x)\rs x$ \\
$\hspace*{5.1cm}$ $=(y\rs y\vee_2 x)\rs (y\rs x)$ \\
$\hspace*{5.1cm}$ $=1\rs (y\rs x)=y\rs x$. \\
It follows that either $y\ra x=1$ or $y\vee_1 x\ra y=1$. Since by assumption $x<y$, we get $y\vee_1 x\ra y=1$,  
that is $y\vee_1 x\le y$. On the other hand $y\le y\vee_1 x$, thus $y\vee_1 x=y$. \\
Similarly, from $(y\vee_2 x\rs y)\rs (y\rs x)=y\rs x$ we get $y\vee_2 x=y$. \\
Applying Theorem \ref{psh-waj-10} it follows that $A$ is a Wajsberg pseudo-hoop. \\
Conversely, let $A$ be a linearly ordered Wajsberg pseudo-hoop and $x, y\in A$ such that $y\ra x=x$.  
If $x\le y$, then from (W1) we get $y=x\rs x=1$. 
If $y\le x$, then obviously $x=1$. \\
Similarly from $y\rs x=x$, it follows that $x=1$ or $y=1$. \\
Hence condition $(A)$ is satisfied, that is $A$ is an Archimedean pseudo-hoop. 
\end{proof}

\begin{cor} \label{psh-waj-130} Any simple pseudo-hoop is a linearly ordered Wajsberg pseudo-hoop. 
\end{cor}
\begin{proof} It follows by Remark \ref{psh-waj-100-20} and Theorem \ref{psh-waj-120}.
\end{proof}

\begin{ex} \label{psh-waj-10-40-10} $\rm($\cite[Ex. 1.6]{Blok1}$\rm)$ 
Let $a\in (0, 1)$ and $C_m=\{a^m, a^{m-1}, \cdots, a, a^0=1\}$ with $m\in {\mathbb N}$. Define the operations: 
$a^k\odot a^n=a^{\min(k+n,m)}$, $a^k\ra a^n=a^{\max(n-k,0)}$. 
Then $(C_m,\odot,\ra,a^m,1)$ is a bounded linearly ordered Wajsberg hoop. 
Indeed, the equality $(a^k\ra a^n)\ra a^n=(a^n\ra a^k)\ra a^k$ is equivalent to 
$a^{\max(n-\max(n-k,0),0)}=a^{\max(k-\max(k-n,0),0)}$. If we consider the cases $n\le k$ and $n>k$, it is easy 
to see that the equality $\max(n-\max(n-k,0),0)=\max(k-\max(k-n,0),0)$ is verified for any $0\le k, n\le m$. 
$C_m$ is a simple Wajsberg hoop (\cite[Ex. 2.4]{Blok1}). A simple computation shows that $C_m$ is an Archimedean 
hoop.  
\end{ex}

$\vspace*{5mm}$

\section{Involutive filters of pseudo-hoops}

In this section we define the involutive filters of a bounded pseudo-hoop and we investigate their properties. 
If $\Den(A)$ is the set of all dense elements of a good pseudo-hoop $A$, we show that $\Den(A)$ is an 
involutive filter of $A$, and any filter of $A$ containing $\Den(A)$ is an involutive filter. 
One of main results consists of proving that a normal filter $F$ of a bounded pseudo-hoop $A$ is involutive if and 
only if $A/F$ is an involutive pseudo-hoop. 
We introduce the notion of a fantastic filter of a pseudo-hoop $A$ and we prove that a normal filter of $A$ 
is fantastic if and only if $A/F$ is a Wajsberg pseudo-hoop. 
We give characterizations of involutive filters of a bounded pseudo-hoop and we prove that in the case of 
bounded Wajsberg pseudo-hoops the notions of fantastic and involutive filters coincide. 
It is also proved that any Boolean filter of a bounded Wajsberg pseudo-hoop is involutive. 

\begin{Def} \label{psh-if-10} If $F\in {\mathcal F}(A)$, then $F$ is said to be an \emph{involutive} filter of $A$ 
if $x^{-\sim} \ra x, x^{\sim-}\rs x\in F$, for all $x\in A$. 
We will denote by ${\mathcal F}_i(A)$ the set of all involutive filters of $A$.  
\end{Def}

Obviously, if $A$ is an involutive pseudo-hoop, then ${\mathcal F}_i(A)={\mathcal F}(A)$. 
In particular, if $A$ is a bounded Wajsberg pseudo-hoop, then ${\mathcal F}_i(A)={\mathcal F}(A)$.

\begin{ex} \label{psh-if-60} If $s\in \mathcal{BS}(A)$, then $\Ker(s)\in {\mathcal F}_i(A)$. \\
Indeed, let $x\in A$. By Proposition \ref{psh-s-20} we have $s(x^{-\sim})=s(x)$. 
Since $x\le x^{-\sim}$, we get $s(x\ra x^{-\sim})=s(1)=1$. 
Applying $(bs_2)$ we get $s(x^{-\sim}\ra x)=1$, that is $x^{-\sim}\ra x\in \Ker(s)$. 
Similarly $x^{\sim-}\rs x\in \Ker(s)$, hence $\Ker(s)\in {\mathcal F}_i(A)$. 
\end{ex}

\begin{ex} \label{psh-if-70} If $m\in \mathcal{M}(A)$, then $\Ker_0(m)\in {\mathcal F}_i(A)$. \\
Indeed, let $x\in A$. Since $x\le x^{-\sim}$, applying Proposition \ref{psh-m-20} we get 
$m(x^{-\sim}\ra x)=m(x)-m(x^{-\sim})=0$ and $m(x^{\sim-}\rs x)=m(x)-m(x^{\sim-})=0$. 
It follows that $x^{-\sim}\ra x, x^{\sim-}\rs x\in \Ker_0(m)$, that is $\Ker_0(m)\in {\mathcal F}_i(A)$.   
\end{ex}

\begin{prop} \label{psh-if-20} If $A$ is a good pseudo-hoop, then $\Den(A)\in {\mathcal F}_n(A)\cap {\mathcal F}_i(A)$.
\end{prop}
\begin{proof}
Let $x, y\in \Den(A)$, that is $x^{-\sim}=y^{-\sim}=1$. Applying Lemma \ref{psh-40} we have 
$(x\odot y)^{-\sim}\ge x^{-\sim}\odot y^{-\sim}=1$, hence $(x\odot y)^{-\sim}=1$, that is $x\odot y\in \Den(A)$. \\ 
If $x\in \Den(A)$ and $y\in A$ such that $x\le y$, then $1=x^{-\sim}\le y^{-\sim}$. \\
Hence $y^{-\sim}=1$, so $y\in \Den(A)$. It follows that $F\in {\mathcal F}(A)$. \\ 
Let $x, y\in A$. Applying Proposition \ref{psh-30}$(2)$ we have: \\
$\hspace*{1cm}$ 
$x\ra y \in \Den(A)$ iff $(x\ra y)^{-\sim}=1$ iff $x^{-\sim}\ra y^{-\sim}=1$ iff $x^{-\sim}\le y^{-\sim}$ \\
$\hspace*{3.9cm}$ iff $x^{-\sim}\rs y^{-\sim}=1$ iff $(x\rs y)^{-\sim}=1$ iff $x\rs y \in \Den(A)$. \\
Thus $F\in {\mathcal F}_n(A)$. \\ 
Let $x\in A$. From Proposition \ref{psh-30}$(1)$ we get $(x^{-\sim}\ra x)^{-\sim}=(x^{-\sim}\rs x)^{-\sim}=1$, 
that is $x^{-\sim}\ra x, x^{\sim-}\rs x\in \Den(A)$. Hence $F\in {\mathcal F}_i(A)$. 
We conclude that $\Den(A)\in {\mathcal F}_n(A)\cap {\mathcal F}_i(A)$.
\end{proof}

\begin{rem} \label{psh-if-20-10} Let $A$ be a good pseudo-hoop and $D=\Den(A)$. Then $(x,y)\in \Theta_D$ iff $x^{-\sim}=y^{-\sim}$. 
Indeed, $(x,y)\in \Theta_D$ iff $x\ra y, y\ra x\in D$ iff $(x\ra y)^{-\sim}=(y\ra x)^{-\sim}=1$ iff 
$x^{-\sim}\ra y^{-\sim}=1$ and $y^{-\sim}\ra x^{-\sim}=1$ iff $x^{-\sim}\le y^{-\sim}$ and $y^{-\sim}\le x^{-\sim}$ 
iff $x^{-\sim}=y^{-\sim}$. 
\end{rem}

\begin{prop} \label{psh-if-20-20} Let $A$ be a normal Wajsberg pseudo-hoop. 
Then $A/\Den(A)$ and $\Inv(A)$ are isomorphic. 
\end{prop}
\begin{proof} Denote $D=\Den(A)$. 
It is obvious that the map $f:A/D\longrightarrow \Inv(A)$, defined by $f([x]_{\Theta_D})=x^{-\sim}$, for all $x\in A$, is a pseudo-hoop isomorphism.  
\end{proof}

\begin{prop} \label{psh-if-30} If $A$ is a good pseudo-hoop, then $\Den(A)\subseteq F$, for any 
$F\in {\mathcal F}_i(A)$. 
\end{prop}
\begin{proof}
Let $F\in {\mathcal F}_i(A)$ and $x\in \Den(A)$, that is $x^{-\sim}=1\in F$ and $x^{-\sim}\ra x\in F$. 
Since $F$ is a filter of $A$, we get $x\in F$, that is $\Den(A)\subseteq F$. 
\end{proof}

\begin{prop} \label{psh-if-40} If $F_1\in {\mathcal F}_i(A)$ and $F_2\in {\mathcal F}(A)$ such that 
$F_1\subseteq F_2$, then $F_2\in {\mathcal F}_i(A)$. 
\end{prop}
\begin{proof} It is straightforward.
\end{proof}

\begin{cor} \label{psh-if-40-10} If $A$ is a good pseudo-hoop and $F\in {\mathcal F}(A)$ such that 
$\Den(A)\subseteq F$, then $F \in {\mathcal F}_i(A)$. 
\end{cor}

The following theorem is a consequence of the above results. 

\begin{theo} \label{psh-if-50} $({\mathcal F}_i(A),\subseteq)$ is a sublattice of the lattice 
$({\mathcal F}(A),\subseteq)$ with the least element $\Den(A)$.
\end{theo}

\begin{theo} \label{psh-if-70-10} A bounded pseudo-hoop $A$ is involutive if and only if $\{1\}\in {\mathcal F}_i(A)$.
\end{theo}
\begin{proof} 
If $A$ is involutive and $x\in A$, then $x^{-\sim}=x^{\sim-}=x$, so $x^{-\sim}\ra x=x^{\sim-}\rs x=1$, that is 
$\{1\}\in {\mathcal F}_i(A)$. \\
Conversely, if $\{1\}\in {\mathcal F}_i(A)$, then we have $x^{-\sim}\ra x=x^{\sim-}\rs x=1$, for all $x\in A$. \\ 
It follows that $x^{-\sim}\le x$ and $x^{\sim-}\le x$, that is $x^{-\sim}=x^{\sim-}=x$. 
Hence $A$ is involutive.  
\end{proof}

\begin{cor} \label{psh-if-70-20} A bounded pseudo-hoop $A$ is involutive if and only if 
${\mathcal F}(A)= {\mathcal F}_i(A)$.
\end{cor}

\begin{theo} \label{psh-if-80} Let $A$ be a bounded pseudo-hoop and let $F\in {\mathcal F}_n(A)$. 
Then $F\in {\mathcal F}_i(A)$ if and only if $A/F$ is an involutive pseudo-hoop. 
\end{theo}
\begin{proof} 
Let $F\in {\mathcal F}_n(A)\cap {\mathcal F}_i(A)$ and let $x\in A$, hence 
$x^{-\sim}\ra x, x^{\sim-}\rs x\in F$. 
Then we have: \\
$\hspace*{2cm}$ $[x^{-\sim}]_{\Theta_F}\ra [x]_{\Theta_F}=[x^{-\sim}\ra x]_{\Theta_F}=[1]_{\Theta_F}$ and \\
$\hspace*{2cm}$ $[x^{\sim-}]_{\Theta_F}\rs [x]_{\Theta_F}=[x^{\sim-}\rs x]_{\Theta_F}=[1]_{\Theta_F}$. \\
It follows that $[x^{-\sim}]_{\Theta_F}=[x^{\sim-}]_{\Theta_F}=[x]_{\Theta_F}$, hence $A/F$ is involutive. \\
Conversely, assume that $A/F$ is an involutive pseudo-hoop and let $x\in A$. \\
We have $[x^{-\sim}]_{\Theta_F}=[x^{\sim-}]_{\Theta_F}=[x]_{\Theta_F}$, thus 
$[x^{-\sim}\ra x]_{\Theta_F}=[x^{\sim-}\rs x]_{\Theta_F}=[1]_{\Theta_F}$. \\
Hence $x^{-\sim}\ra x, x^{\sim-}\rs x\in F$, that is $F\in {\mathcal F}_i(A)$. 
\end{proof}

\begin{cor} \label{psh-if-80-10} If $A$ is a good pseudo-hoop, then $A/\Den(A)$ is an involutive pseudo-hoop. 
\end{cor}

\begin{theo} \label{psh-if-90} Let $A$ be a good pseudo-hoop and $F\in {\mathcal F}(A)$. 
The following are equivalent:\\ 
$(a)$ $F\in {\mathcal F}_i(A);$ \\
$(b)$ $(y^{-}\rs x^{-})\ra (x\ra y)\in F$ and $(y^{\sim}\ra x^{\sim})\rs (x\rs y)\in F$, for all $x, y\in A;$ \\
$(c)$ $(x^{-}\rs y)\ra (y^{\sim}\ra x)\in F$ and $(x^{\sim}\ra y)\rs (y^{-}\rs x)\in F$, for all $x, y\in A$. 
\end{theo}
\begin{proof}
$(a)\Rightarrow (b)$ 
Let $F\in {\mathcal F}_i(A)$ and $x, y\in A$. Then: \\
$\hspace*{2cm}$ $(x\ra y)^{-\sim}\ra (x\ra y)\in F$ and $(x\rs y)^{-\sim}\rs (x\rs y)\in F$. \\
Applying Propositions \ref{psh-30}$(2)$ we have: \\
$\hspace*{2cm}$ $(x^{-\sim}\ra y^{-\sim})\ra (x\ra y)\in F$ and $(x^{-\sim}\rs y^{-\sim})\rs (x\rs y)\in F$. \\
Hence by Propositions \ref{psh-20-10}$(4)$ we get: \\
$\hspace*{2cm}$ $(y^{-}\rs x^{-})\ra (x\ra y)\in F$ and $(y^{\sim}\ra x^{\sim})\rs (x\rs y)\in F$. \\ 
$(b)\Rightarrow (c)$
Replacing $x$ by $x^{\sim}$ and $x$ by $x^{-}$ in $(b)$ we obtain: \\
$\hspace*{2cm}$ $(y^{-}\rs x^{-\sim})\ra (x^{\sim}\ra y)\in F$ and $(y^{\sim}\ra x^{-\sim})\rs (x^{-}\rs y)\in F$, \\
respectively. Changing $x$ and $y$ we get: \\
$\hspace*{2cm}$ $(x^{-}\rs y^{-\sim})\ra (y^{\sim}\ra x)\in F$ and $(x^{\sim}\ra y^{-\sim})\rs (y^{-}\rs x)\in F$. \\
From $y\le y^{-\sim}$, applying Proposition \ref{psh-20}$(6)$ we have: \\ 
$\hspace*{2cm}$ $x^{-}\rs y\le x^{-}\rs y^{-\sim}$ and $x^{\sim}\ra y\le x^{\sim}\ra y^{-\sim}$. \\
Finally, by Proposition \ref{psh-20}$(7)$ we get: \\
$\hspace*{2cm}$ $(x^{-}\rs y^{-\sim})\ra (y^{\sim}\ra x)\le (x^{-}\rs y)\ra (y^{\sim}\ra x)$ and \\
$\hspace*{2cm}$ $(x^{\sim}\ra y^{-\sim})\rs (y^{-}\rs x)\le (x^{\sim}\ra y)\rs (y^{-}\rs x)$. \\ 
It follows that $(x^{-}\rs y)\ra (y^{\sim}\ra x)\in F$ and $(x^{\sim}\ra y)\rs (y^{-}\rs x)\in F$. \\
$(c)\Rightarrow (a)$ Taking $y:=0$ in $(c)$ 
we get $x^{-\sim}\ra x, x^{-\sim}\rs x\in F$. 
Hence $F\in {\mathcal F}_i(A)$. 
\end{proof}

\begin{prop} \label{psh-if-100} Let $f: A\longrightarrow B$ be a bounded pseudo-hoop homomorphism. 
If $F\in {\mathcal F}_i(B)$, then $f^{-1}(F)\in {\mathcal F}_i(A)$. 
\end{prop}
\begin{proof}
Consider $F\in {\mathcal F}_i(B)$ and let $x\in A$. Then $f(x)=y\in B$. 
Since $F\in {\mathcal F}_i(B)$, then $y^{-\sim}\ra y, y^{\sim-}\rs y \in F$, so
$f(x)^{-\sim}\ra f(x), f(x)^{\sim-}\rs f(x) \in F$. \\
Hence $f(x^{-\sim}\ra x), f(x^{\sim-}\rs x) \in F$, that is $x^{-\sim}\ra x, x^{\sim-}\rs x \in f^{-1}(F)$. \\
We conclude that $f^{-1}(F)\in {\mathcal F}_i(A)$. 
\end{proof}

\begin{Def} \label{psh-ff-10} Let $(A,\odot,\ra,\rs,1)$ be a pseudo-hoop and let $F\in {\mathcal F}(A)$. 
Then $F$ is called \emph{fantastic} if it satisfies the following conditions 
for all $x, y\in A$: \\ 
$(ff_1)$ $y\ra x\in F$ implies $x\vee_1 y\ra x\in F;$ \\
$(ff_2)$ $y\rs x\in F$ implies $x\vee_2 y\rs x\in F$. 
\end{Def}

We will denote by ${\mathcal F}_f(A)$ the set of all fantastic filters of a pseudo-hoop $A$. 

\begin{prop} \label{psh-ff-20} Let $(A,\ra,\rs,1)$  be a pseudo-hoop and let $F\subseteq A$. 
Then $F\in {\mathcal F}_f(A)$ if and only if it satisfies the following conditions for all $x, y, z\in A$: \\
$(1)$ $1\in F;$ \\
$(2)$ $z\ra (y\ra x)\in F$ and $z\in F$ implies $x\vee_1 y\ra x\in F;$ \\
$(3)$ $z\rs (y\rs x)\in F$ and $z\in F$ implies $x\vee_2 y\rs x\in F$.
\end{prop}
\begin{proof}
Similarly as \cite[Prop. 4.2]{Ciu21}.
\end{proof}

\begin{theo} \label{psh-ff-20-10} A pseudo-hoop $A$ is a Wajsberg pseudo-hoop if and only if  
$\{1\}\in {\mathcal F}_f(A)$. 
\end{theo}
\begin{proof}
Similarly as \cite[Th. 4.7]{Ciu21}.
\end{proof}

\begin{theo} \label{psh-ff-30} Let $A$ be a pseudo-hoop and $F\in {\mathcal F}_n(A)$. 
Then $F\in {\mathcal F}_f(A)$ if and only if $A/F$ is a Wajsberg pseudo-hoop. 
\end{theo}
\begin{proof}
Let $F\in {\mathcal F}_f(A)$ and $x, y\in A$ such that $[y]_{\Theta_F}\ra [x]_{\Theta_F}=[1]_{\Theta_F}$, 
so $[y\ra x]_{\Theta_F}=[1]_{\Theta_F}$, that is $y\ra x\in F$. 
Hence $x\vee_1 y\ra x\in F$, so 
$([x]_{\Theta_F}\vee_1 [y]_{\Theta_F})\ra [x]_{\Theta_F}= [x\vee_1 y\ra x]_{\Theta_F}=[1]_{\Theta_F}$. 
Thus $[1]_{\Theta_F}\in {\mathcal F}_f(A/F)$, so by Theorem \ref{psh-ff-20-10}, $A/F$ is a a Wajsberg pseudo-hoop. \\
Conversely, if $A/F$ is a Wajsberg pseudo-hoop, then $[1]_{\Theta_F}\in {\mathcal F}_f(A/F)$. 
If $y\ra x\in F=[1]_{\Theta_F}$, we have $[y]_{\Theta_F}\ra [x]_{\Theta_F}\in [1]_{\Theta_F}$. 
Since $[1]_{\Theta_F}\in {\mathcal F}_f(A/F)$, we get 
$([x]_{\Theta_F}\vee_1 [y]_{\Theta_F})\ra [x]_{\Theta_F}= [1]_{\Theta_F}$, so 
$[x\vee_1 y\ra x]_{\Theta_F}=[1]_{\Theta_F}$, that is $x\vee_1 y\ra x\in F$. \\
Similarly $y\rs x\in F$ implies $x\vee_2 y\rs x\in F$. 
It follows that $F\in {\mathcal F}_f(A)$.
\end{proof}

\begin{prop} \label{psh-ff-40} If $A$ is a bounded pseudo-hoop, then ${\mathcal F}_f(A)\subseteq {\mathcal F}_i(A)$. 
\end{prop}
\begin{proof}
Let $F\in {\mathcal F}_f(A)$ and $x\in A$. Since $0\ra x=1\in F$ and $0\rs x=1\in F$, then 
$x\vee_1 0\ra x\in F$ and $x\vee_2 0\rs x\in F$, that is $x^{-\sim}\ra x, x^{\sim-}\rs x \in F$. 
It follows that $F\in {\mathcal F}_i(A)$, hence ${\mathcal F}_f(A)\subseteq {\mathcal F}_i(A)$.
\end{proof}

The next result is proved following an idea from \cite{Rac3}. 

\begin{theo} \label{psh-ff-50} If $A$ is a good pseudo-hoop, then ${\mathcal F}_f(A)= {\mathcal F}_i(A)$. 
\end{theo}
\begin{proof}
According to Proposition \ref{psh-ff-40}, ${\mathcal F}_f(A)\subseteq {\mathcal F}_i(A)$. \\
Conversely, let $F\in {\mathcal F}_i(A)$, that is $x^{-\sim}\ra x, x^{\sim-}\rs x \in F$, for all $x\in A$. \\
Consider $x, y\in A$ such that $y\ra x\in F$. By Proposition \ref{psh-20-10}$(6)$, $y\ra x\le x^{-}\rs y^{-}$, 
hence $x^{-}\rs y^{-}\in F$. 
Applying Proposition \ref{psh-20}$(8)$,$(2)$,$(10)$ we have: \\ 
$\hspace*{2cm}$ $x^{-}\rs y^{-}=x^{-}\rs (x^{-}\wedge y^{-})$ \\
$\hspace*{3.6cm}$ $=x^{-}\rs (y^{-}\odot (y^{-}\rs x^{-}))$ \\
$\hspace*{3.6cm}$ $=x^{-}\rs (y^{-}\odot (y^{-}\rs (x\ra 0)))$ \\
$\hspace*{3.6cm}$ $=x^{-}\rs (y^{-}\odot (x\ra (y^{-}\rs 0)))$ \\
$\hspace*{3.6cm}$ $=x^{-}\rs (y^{-}\odot (x\ra y^{-\sim}))$. \\
It follows that $y\ra x\le x^{-}\rs (y^{-}\odot (x\ra y^{-\sim}))$, hence 
$x^{-}\rs (y^{-}\odot (x\ra y^{-\sim}))\in F$. \\
Using Propositions \ref{psh-20}$(11)$,$(12)$, \ref{psh-20-10}$(4)$ and \ref{psh-30}$(2)$ we get: \\
$\hspace*{0.5cm}$ $(x^{-}\rs (y^{-}\odot (x\ra y^{-\sim})))\rs (x^{-}\rs (y^{-}\odot (x\ra y)))$ \\ 
$\hspace*{5.1cm}$ $\ge (y^{-}\odot (x\ra y^{-\sim}))\rs (y^{-}\odot (x\ra y))$ \\  
$\hspace*{5.1cm}$ $\ge (x\ra y^{-\sim})\rs (x\ra y)$ \\
$\hspace*{5.1cm}$ $=(x^{-\sim}\ra y^{-\sim})\rs (x\ra y)$ \\
$\hspace*{5.1cm}$ $=(x\ra y)^{-\sim}\rs (x\ra y)$. \\ 
Since $F$ is involutive, $(x\ra y)^{-\sim}\rs (x\ra y)\in F$, hence \\ 
$\hspace*{2cm}$ $(x^{-}\rs (y^{-}\odot (x\ra y^{-\sim})))\rs (x^{-}\rs (y^{-}\odot (x\ra y)))\in F$. \\
It follows that $x^{-}\rs (y^{-}\odot (x\ra y))\in F$. 
Now, by Proposition \ref{psh-20-10}$(6)$,$(5)$ we have: \\
$\hspace*{1cm}$ $x^{-}\rs (y^{-}\odot (x\ra y))\le ((y^{-}\odot (x\ra y))^{\sim}\ra x^{-\sim} 
                 =((x\ra y)\rs y^{-\sim})\ra x^{-\sim}$. \\
Thus $((x\ra y)\rs y^{-\sim})\ra x^{-\sim}\in F$. \\
Applying Proposition \ref{psh-20}$(5)$, \ref{psh-20-10}$(4)$ and \ref{psh-30}$(2)$ we get: \\ 
$\hspace*{1cm}$ $(((x\ra y)\rs y^{-\sim})\ra x^{-\sim})\rs (((x\ra y)\rs y)\ra x^{-\sim})$ \\
$\hspace*{5.7cm}$ $\ge ((x\ra y)\rs y)\ra ((x\ra y)\rs y^{-\sim})$ \\
$\hspace*{5.7cm}$ $=((x\ra y)\rs y)\ra ((x\ra y)^{-\sim}\rs y^{-\sim})$ \\
$\hspace*{5.7cm}$ $=((x\ra y)\rs y)\ra ((x\ra y)\rs y)^{-\sim}=1\in F$. \\
It follows that $((x\ra y)\rs y)\ra x^{-\sim}\in F$. 
Finally, by Proposition \ref{psh-20}$(11)$ we get: \\
$\hspace*{1cm}$ $(((x\ra y)\rs y)\ra x^{-\sim})\ra (((x\ra y)\rs y)\ra x)\ge x^{-\sim}\ra x\in F$. \\ 
Hence $((x\ra y)\rs y)\ra x =x\vee_1 y\ra x\in F$. 
Similarly, $y\rs x\in F$ implies $x\vee_2 y\rs x\in F$. \\ 
It follows that $F\in {\mathcal F}_f(A)$, that is ${\mathcal F}_i(A\subseteq {\mathcal F}_f(A)$. 
We conclude that ${\mathcal F}_f(A)= {\mathcal F}_i(A)$.
\end{proof}

\begin{cor} \label{psh-ff-60} If $A$ is a bounded Wajsberg pseudo-hoop, then ${\mathcal F}_f(A)= {\mathcal F}_i(A)$. 
\end{cor}

\begin{cor} \label{psh-ff-60-10} If $A$ is a bounded hoop, then ${\mathcal F}_f(A)= {\mathcal F}_i(A)$. 
\end{cor}

\begin{Def} \label{psh-bf-10} Let $(A,\ra,\rs,0,1)$ be a bounded Wajsberg pseudo-hoop and let $F\in {\mathcal F}(A)$. 
Then $F$ is called a \emph{Boolean} filter if $x\vee x^{-}, x\vee x^{\sim}\in F$, for all $x\in A$. 
\end{Def}

We will denote by ${\mathcal F}_B(A)$ the set of all Boolean filters of a bounded Wajsberg pseudo-hoop $A$. 

\begin{prop} \label{psh-bf-20} If $A$ is a bounded Wajsberg pseudo-hoop, then 
${\mathcal F}_B(A)\subseteq {\mathcal F}_i(A)$. 
\end{prop}
\begin{proof}
Let $F\in {\mathcal F}_B(A)$ and $x\in A$, that is $x\vee x^{-}, x\vee x^{\sim}\in F$. 
From $x^{-\sim}=x^{-}\rs 0\le x^{-}\rs x$, we get 
$x\vee x^{-}=x^{-}\vee x=x^{-}\vee_2 x=(x^{-}\rs x)\ra x\le x^{-\sim}\ra x$. Hence $x^{-\sim}\ra x\in F$. 
Similarly $x^{\sim-}\rs x\in F$, thus $F\in {\mathcal F}_i(A)$. 
We conclude that ${\mathcal F}_B(A)\subseteq {\mathcal F}_i(A)$.
\end{proof}

\begin{ex} \label{psh-ff-70} Consider the set $A=\{0,a,b,c,1\}$ and the operations $\odot,\ra$ given by 
the following tables:
\[
\hspace{10mm}
\begin{array}{c|ccccccc}
\odot & 0 & a & b & c & 1 \\ \hline
0 & 0 & 0 & 0 & 0 & 0 \\ 
a & 0 & a & 0 & a & a \\ 
b & 0 & 0 & b & b & b \\ 
c & 0 & a & b & c & c \\ 
1 & 0 & a & b & c & 1 
\end{array}
\hspace{10mm} 
\begin{array}{c|ccccccc}
\ra & 0 & a & b & c & 1 \\ \hline
0 & 1 & 1 & 1 & 1 & 1 \\ 
a & b & 1 & b & 1 & 1 \\ 
b & a & a & 1 & 1 & 1 \\ 
c & 0 & a & b & 1 & 1 \\ 
1 & 0 & a & b & c & 1 
\end{array}
. 
\]
Then $(A,\odot,\ra,0,1)$ is a bounded hoop. 
Since $0\vee c=c\neq 1=c\vee 0$, it follows that $A$ is not a Wajsberg hoop. 
One can see that: \\
$\hspace*{2cm}$
${\mathcal F}(A)={\mathcal F}_n(A)=\{\{1\},\{c,1\},\{a,c,1\},\{b,c,1\},\{a,b,c,1\},A\}$, \\
$\hspace*{2cm}$
${\mathcal F}_f(A)={\mathcal F}_i(A)=\{\{c,1\},\{a,c,1\},\{b,c,1\},\{a,b,c,1\},A\}$.
\end{ex}

\begin{ex} \label{psh-ff-80} Consider the set $A=\{0,a,b,c,1\}$ and the operations $\odot,\ra$ given by 
the following tables:
\[
\hspace{10mm}
\begin{array}{c|ccccccc}
\odot & 0 & a & b & c & 1 \\ \hline
0 & 0 & 0 & 0 & 0 & 0 \\ 
a & 0 & 0 & 0 & 0 & a \\ 
b & 0 & 0 & 0 & a & b \\ 
c & 0 & 0 & a & b & c \\ 
1 & 0 & a & b & c & 1 
\end{array}
\hspace{10mm} 
\begin{array}{c|ccccccc}
\ra & 0 & a & b & c & 1 \\ \hline
0 & 1 & 1 & 1 & 1 & 1 \\ 
a & c & 1 & 1 & 1 & 1 \\ 
b & b & c & 1 & 1 & 1 \\ 
c & a & b & c & 1 & 1 \\ 
1 & 0 & a & b & c & 1 
\end{array}
. 
\]
Then $(A,\odot,\ra,0,1)$ is a bounded Wajsberg hoop, isomorphic to the hoop $(C_4,\odot,\ra,a^4,1)$ from 
Example \ref{psh-waj-10-40-10}. Then we have  
${\mathcal F}(A)={\mathcal F}_f(A)={\mathcal F}_i(A)=\{\{1\},A\}$ and ${\mathcal F}_B(A)=\{A\}$.
\end{ex}

$\vspace*{5mm}$

\section{State pseudo-hoops}

In this section we introduce the notion of state operators on pseudo-hoops. 
We define three types of state operators on pseudo-hoops: type I and type II as generalization of 
state operators on hoops from \cite{Bor2}, and type III as generalization of state operators on bounded pseudo-hoops from \cite{Ciu22}, \cite{Ciu23}. 
We prove that a pseudo-hoop is Wajsberg if and only if the type I and type II state operators coincide. 
For the case of a bounded pseudo-hoop it is proved that the kernel of a type II state operator is an involutive filter. Moreover, for a bounded Wajsberg pseudo-hoop the kernel of any type of state operator is an involutive filter. 
As main results we show that any type II state operator is a type III state operator, and in the case of bounded 
Wajsberg pseudo-hoops any type I state operator is a type III state operator. 
If the kernel of a type II state operator $\mu$ is a normal filter, then it is proved that $\mu$ is also a type I 
state operator. 

\begin{Def} \label{psh-st-10} Let $(A,\odot,\ra,\rs,1)$ be a pseudo-hoop and $\mu:A \longrightarrow A$ be a unary operator on $A$. For all $x, y\in A$ consider the following axioms:\\
$(IS_1)$ $\mu(x\ra y)=\mu(x\vee_1 y)\ra \mu(y)$ and $\mu(x\rs y)=\mu(x\vee_2 y)\rs \mu(y),$ \\
$(IS^{'}_1)$ $\mu(x\ra y)=\mu(y\vee_1 x)\ra \mu(y)$ and $\mu(x\rs y)=\mu(y\vee_2 x)\rs \mu(y),$ \\ 
$(IS^{''}_1)$ $\mu(x\ra y)=\mu(x)\ra \mu(x\wedge y)$ and $\mu(x\rs y)=\mu(x)\rs \mu(x\wedge y),$ \\ 
$(IS_2)$ $\mu(x\odot y)=\mu(x)\odot \mu(x\rs x\odot y)=\mu(y\ra x\odot y)\odot \mu(y),$ \\
$(IS_3)$ $\mu(\mu(x)\odot \mu(y))=\mu(x)\odot \mu(y),$ \\
$(IS_4)$ $\mu(\mu(x)\ra \mu(y))=\mu(x)\ra \mu(y)$ and 
         $\mu(\mu(x)\rs \mu(y))=\mu(x)\rs \mu(y)$. \\
Then: \\
$(i)$ $\mu$ is called an \emph{internal state of type I} or a \emph{state operator of type I} or 
a \emph{type I state operator} if it satisfies axioms $(IS_1)$, $(IS_2)$, $(IS_3)$, $(IS_4);$ \\    
$(ii)$ $\mu$ is called an \emph{internal state of type II} or a \emph{state operator of type II} or a 
\emph{type II state operator} if it satisfies axioms $(IS^{'}_1)$, $(IS_2)$, $(IS_3)$, $(IS_4)$. \\
$(iii)$ $\mu$ is called an \emph{internal state of type III} or a \emph{state operator of type III} or a 
\emph{type III state operator} if it satisfies axioms $(IS^{''}_1)$, $(IS_2)$, $(IS_3)$, $(IS_4)$. \\
The structure $(A, \odot, \ra, \rs, \mu, 1)$ ($(A,\mu)$, for short) is called a 
\emph{state pseudo-hoop of type I (type II, type III) state pseudo-hoop}, respectively.
\end{Def}

Denote $\mathcal{IS}^{(I)}(A)$, $\mathcal{IS}^{(II)}(A)$, and $\mathcal{IS}^{(III)}(A)$ the set of all internal states of type I, II and III on a pseudo-hoop $A$, respectively. For a bounded pseudo-hoop $(A,\odot,\ra,\rs,0,1)$ we denote by $\mathcal{IS}^{(I)}_1(A)$, $\mathcal{IS}^{(II)}_1(A)$, and $\mathcal{IS}^{(III)}_1(A)$ the set of all internal states $\mu$ from $\mathcal{IS}^{(I)}(A)$, $\mathcal{IS}^{(II)}(A)$, and $\mathcal{IS}^{(III)}(A)$ such that $\mu(0)=0$, respectively. \\ 
For an internal state $\mu$, $\Ker(\mu)=\{x\in A \mid \mu(x)=1\}$  is called the \emph{kernel} of $\mu$. \\
An internal state is called \emph{faithful} if $\Ker(\mu)=1$. 

\begin{ex} \label{psh-st-20} Let $(A, \odot, \ra, \rs, 1)$ be a pseudo-hoop and 
$1_A, Id_A:A\longrightarrow A$, defined by $1_A(x)=1$ and $Id_A(x)=x$ for all $x\in A$. 
Obviously $1_A\in \mathcal{IS}^{(I)}(A)\cap \mathcal{IS}^{(II)}(A)\cap \mathcal{IS}^{(III)}(A)$. 
By Propositions \ref{psh-90}, \ref{psh-20-05-10}, and \ref{psh-20}$(8)$,  
$Id_A\in \mathcal{IS}^{(I)}(A)\cap \mathcal{IS}^{(III)}(A)$. 
\end{ex}

\begin{ex} \label{psh-st-30} Let $(A_1, \odot_1, \ra_1, \rs_1, 1_1)$ and 
$(A_2, \odot_2, \ra_2, \rs_2, 1_2)$ be two pseudo-hoops.
Denote $A=A_1 \times A_2=\{(x_1,x_2) \mid x_1\in A_1, x_2\in A_2\}$ and for all $(x_1, x_2), (y_1, y_2)\in A$, 
define the operations $\ra, \rs, 1$ as follows: 
$(x_1, x_2)\odot (y_1, y_2)=(x_1\odot_1 y_1, x_2\odot_2 y_2)$,
$(x_1, x_2)\ra (y_1, y_2)=(x_1\ra_1 y_1, x_2\ra_2 y_2)$, 
$(x_1, x_2)\rs (y_1, y_2)=(x_1\rs_1 y_1, x_2\rs_2 y_2)$, 
$1=(1_1, 1_2)$. 
Obviously $(A, \odot, \ra, \rs, 1)$ is a pseudo-hoop. 
Consider $\mu_1\in \mathcal{IS}^{(I)}(A_1)$, $\mu_2\in \mathcal{IS}^{(I)}(A_2)$ and define the map $\mu:A\longrightarrow A$ by $\mu((x,y))=(\mu_1(x),\mu_2(x))$. Then $\mu\in \mathcal{IS}^{(I)}(A)$.  
Similarly if $\mu_1$ and $\mu_2$ are type II or type III internal states.
\end{ex}

\begin{prop} \label{psh-st-40} A pseudo-hoop $A$ is a Wajsberg pseudo-hoop if and only if $\mathcal{IS}^{(I)}(A)=\mathcal{IS}^{(II)}(A)$. 
\end{prop}
\begin{proof}
It is clear that if $A$ is a Wajsberg pseudo-hoop then $\mathcal{IS}^{(I)}(A)=\mathcal{IS}^{(II)}(A)$. 
Conversely, suppose that $\mathcal{IS}^{(I)}(A)=\mathcal{IS}^{(II)}(A)$. 
Since $Id_A\in \mathcal{IS}^{(I)}(A)$, we have $Id_A\in \mathcal{IS}^{(II)}(A)$, so 
$x\ra y=y\vee_1 x\ra y$ and $x\rs y=y\vee_2 x\rs y$, for all $x, y\in A$.  
According to Theorem \ref{psh-waj-10}, it follows that $A$ is a Wajsberg pseudo-hoop. 
\end{proof}

\begin{prop} \label{psh-st-50} Let $(A, \odot, \ra, \rs, \mu, 1)$ be a type I, type II or a type III state 
pseudo-hoop. Then the following hold:\\ 
$(1)$ $\mu(1)=1;$ \\
$(2)$ if $x\le y$, then $\mu(x)\le \mu(y);$ \\
$(3)$ $\mu(\mu(x))=\mu(x)$, for all $x\in A;$ \\
$(4)$ $\mu(x\ra y)\le \mu(x)\ra \mu(y)$ and 
      $\mu(x\rs y)\le \mu(x)\rs \mu(y)$, for all $x, y\in A;$ \\
$(5)$ $\mu(x)\odot \mu(y)\le \mu(x \odot y)$ and $\mu(x)\odot \mu(y)\le \mu(x \wedge y)$, for all $x, y\in A;$ \\      
$(6)$ $\Ker(\mu)\in {\mathcal F}(A);$ \\
$(7)$ $\Img(\mu)$ is a subalgebra of $A;$ \\ 
$(8)$ $\Img(\mu)=\{x\in A \mid x=\mu(x)\};$ \\
$(9)$ $\Ker(\mu)\cap \Img(\mu)=\{1\}$. 
\end{prop}
\begin{proof}
$(1)$ From $(IS_2)$, $(IS^{'}_2)$ and $(IS^{''}_2)$ for $x=y=1$ we get $\mu(1)=\mu(1)\ra \mu(1)=1$. \\
$(2)$ By Proposition \ref{psh-20}$(2)$, $x=y\wedge x=(y\ra x)\odot y=y\odot(y\rs x)$. 
Applying $(IS_2)$ we get $\mu(x)=\mu(y\odot (y\rs x))=\mu(y)\odot \mu(y\rs y\odot (y\rs x))\le \mu(y)$. \\
$(3)$ Applying $(1)$ and $(IS_3)$ we have: \\
$\hspace*{1cm}$
$\mu(\mu(x))=\mu(1\ra \mu(x))=\mu(\mu(1)\ra \mu(x))=\mu(1)\ra \mu(x)=1\ra \mu(x)=\mu(x)$. \\
$(4)$ If $\mu\in \mathcal{IS}^{(I)}(A)$ then from $x\le (x\ra y)\rs y$ we 
get $\mu(x)\le \mu((x\ra y)\rs y)$, so 
$\mu((x\ra y)\rs y)\ra \mu(y)\le \mu(x)\ra \mu(y)$, that is $\mu(x\ra y)\le \mu(x)\ra \mu(y)$. \\
Similarly $\mu(x\rs y)\le \mu(x)\rs \mu(y)$. \\
If $\mu\in \mathcal{IS}^{(II)}(A)$ then from $x\le (y\ra x)\rs x$ we have 
$\mu(x)\le \mu((y\ra x)\rs x)$ and in a similar way we get 
$\mu(x\ra y)\le \mu(x)\ra \mu(y)$ and $\mu(x\rs y)\le \mu(x)\rs \mu(y)$. \\
For $\mu\in \mathcal{IS}^{(III)}(A)$, since $\mu(x\wedge y)\le \mu(y)$, we have: \\  
$\hspace*{2cm}$ $\mu(x\ra y)=\mu(x)\ra \mu(x\wedge y)\le \mu(x)\ra \mu(y)$ and \\
$\hspace*{2cm}$ $\mu(x\rs y)=\mu(x)\ra \mu(x\wedge y)\le \mu(x)\rs \mu(y)$. \\
$(5)$ From $y\le x\rs x\odot y$ we get $\mu(y)\le \mu(x\rs x\odot y)$, so that: \\
$\hspace*{2cm}$ $\mu(x)\odot \mu(y)\le \mu(x)\odot \mu(x\rs x\odot y)=\mu(x\odot y)$. \\
Since $x\odot y\le x\wedge y$, we have $\mu(x)\odot \mu(y)\le \mu(x\odot y)\le \mu(x\wedge y)$. \\
$(6)$ Consider $x, x\ra y\in \Ker(\mu)$, that is $\mu(x)=\mu(x\ra y)=1$. \\
Applying $(4)$ we have $1=\mu(x\ra y)\le \mu(x)\ra \mu(y)$, so $\mu(x)\ra \mu(y)=1$.  \\
It follows that $\mu(y)=1\ra \mu(y)=\mu(1)\ra \mu(y)=\mu(x)\ra \mu(y)=1$, 
hence  $y\in \Ker(\mu)$. \\
Since $1\in \Ker(\mu)$, it follows that $\Ker(\mu)\in {\mathcal F}(A)$. \\
$(7)$ Since $1=\mu(1)\in \Img(\mu)$, we have $1\in \Img(\mu)$. \\
If $x, y\in \Img(\mu)$ then from $\mu(x)\ra \mu(y)=\mu(\mu(x)\ra \mu(y))$, $\mu(x)\rs \mu(y)=\mu(\mu(x)\rs \mu(y))$ 
and $\mu(x)\odot \mu(y)=\mu(\mu(x)\odot \mu(y))$, it follows that 
$\mu(x)\ra \mu(y), \mu(x)\rs \mu(y), \mu(x)\odot \mu(y)\in \Img(\mu)$. 
Thus $\Img(\mu)$ is a subalgebra of $A$. \\
$(8)$ Clearly $\{x\in A \mid x=\mu(x)\} \subseteq \Img(\mu)$. 
Let $x\in \Img(\mu)$, that is there exists $x_1\in A$ such that $x=\mu(x_1)$. 
It follows that $x=\mu(x_1)=\mu(\mu(x_1))=\mu(x)$, that is $x\in \Img(\mu)$. 
Thus $\Img(\mu) \subseteq \{x\in A \mid x=\mu(x)\}$ and we  conclude that $\Img(\mu)=\{x\in A \mid x=\mu(x)\}$. \\
$(9)$ Let $y\in \Ker(\mu)\cap \Img(\mu)$, so $\mu(y)=1$ and there exists $x\in A$ such that $\mu(x)=y$. \\
It follows that $y=\mu(x)=\mu(\mu(x))=\mu(y)=1$, thus $\Ker(\mu)\cap \Img(\mu)=\{1\}$.                
\end{proof}

\begin{rem} \label{psh-st-50-05}
$(1)$ If $(A, \odot, \ra, 1)$ is a hoop, then the states operators of types I and II coincide with the state operators  
on hoops defined in \cite{Bor2}. We mention that in this definition the order-preserving condition is superfluous. \\ 
$(2)$ If $(A, \odot, \ra, \rs, 0, 1)$ is a bounded pseudo-hoop, then the type III state operator coincides with 
the state operator on bounded pseudo-hoops studied in \cite{Ciu22}.
\end{rem}

\begin{prop} \label{psh-st-50-10} Let $(A, \odot, \ra, \rs, \mu, 1)$ be a type I, type II or a type III state 
pseudo-hoop. Then the following hold for all $x, y\in A$:\\ 
$(1)$ $\mu(x\wedge y)=\mu(x)\odot \mu(x\rs y)=\mu(y\ra x)\odot \mu(y);$ \\
$(2)$ $\mu(\mu(x)\wedge \mu(x))=\mu(x)\wedge \mu(y);$ \\
$(3)$ $\mu(\mu(x)\vee_1 \mu(y))=\mu(x)\vee_1 \mu(y)$ and $\mu(\mu(x)\vee_2 \mu(y))=\mu(x)\vee_2 \mu(y);$ \\
$(4)$ $\mu(x\vee_1 y)\le \mu(x)\vee_1 \mu(y)$ and $\mu(x\vee_2 y)\le \mu(x)\vee_2 \mu(y);$ \\
$(5)$ if $A$ is a Wajsberg pseudo-hoop then $\mu(x\ra y)\rs \mu(y)=\mu(x\rs y)\ra \mu(y);$ \\
$(6)$ if $A$ is cancellative then $\mu(x\odot y)=\mu(x)\odot \mu(y)$.  
\end{prop}
\begin{proof}
$(1)$ Applying Proposition \ref{psh-20}$(2)$,$(8)$ and $(IS_2)$ we get: \\
$\hspace*{2cm}$ $\mu(x\wedge y)= \mu(x\odot (x\rs y))=\mu(x)\odot (x\rs x\odot (x\rs y))$ \\
$\hspace*{3.5cm}$ $=\mu(x)\odot \mu(x\rs x\wedge y)=\mu(x)\odot \mu(x\rs y)$, \\
$\hspace*{2cm}$ $\mu(x\wedge y)= \mu((y\ra x)\odot y)=\mu(y\ra (y\ra x)\odot y)\odot \mu(y)$ \\
$\hspace*{3.5cm}$ $=\mu(y\ra x\wedge y)\odot \mu(y)=\mu(y\ra x)\odot \mu(y)$. \\
$(2)$ By $(1)$, Propositions \ref{psh-st-50}$(3)$, \ref{psh-20}$(2)$ and $(IS_4)$ we have: \\ 
$\hspace*{2cm}$ $\mu(\mu(x)\wedge \mu(y))=\mu(\mu(x)\odot (\mu(x)\rs \mu(y))=\mu^2(x)\odot \mu(\mu(x)\rs \mu(y))$ \\
$\hspace*{4.7cm}$ $=\mu(x)\odot (\mu(x)\rs \mu(y))=\mu(x)\wedge \mu(y)$. \\
$(3)$ Using Propositions \ref{psh-st-50}$(3)$ and $(IS_4)$ we get: \\
$\hspace*{2cm}$ $\mu(\mu(x)\vee_1 \mu(y))=\mu((\mu(x)\ra \mu(y))\rs \mu(y))$ \\
$\hspace*{4.7cm}$ $=\mu((\mu(\mu(x)\ra \mu(y)))\rs \mu(y))$ \\
$\hspace*{4.7cm}$ $=(\mu(x)\ra \mu(y))\rs \mu(y)=\mu(x)\vee_1 \mu(y)$. \\
Similarly $\mu(\mu(x)\vee_2 \mu(y))=\mu(x)\vee_2 \mu(y)$. \\
$(4)$ It follows by Proposition \ref{psh-st-50}$(4)$,$(3)$. \\ 
$(5)$ If $A$ is a Wajsberg pseudo-hoop then it is a lattice, where \\ 
$\hspace*{2cm}$ $x\vee y=(x\ra y)\rs y=(x\rs y)\ra y$. \\
Applying Proposition \ref{psh-90} we have: \\
$\hspace*{1cm}$
$\mu(x\vee y)=\mu((x\ra y)\rs y)=
              \mu(((x\ra y)\rs y)\ra y)\rs \mu(y)$ \\
$\hspace*{5.7cm}$              
              $=\mu(x\ra y)\rs \mu(y)$, \\
$\hspace*{1cm}$
$\mu(x\vee y)=\mu((x\rs y)\ra y)=
              \mu(((x\rs y)\ra y)\rs y)\ra \mu(y)$ \\
$\hspace*{5.7cm}$               
              $=\mu(x\rs y)\ra \mu(y)$. \\
Hence $\mu(x\ra y)\rs \mu(y)=\mu(x\rs y)\ra \mu(y)$. \\
$(6)$ Since $A$ is cancellative, by Proposition \ref{psh-60-10}, $y\ra x\odot y=x$ and applying $(IS_2)$ we get 
$\mu(x\odot y)=\mu(x)\odot \mu(y)$. 
\end{proof}

\begin{prop} \label{psh-st-50-20} Let $(A, \odot, \ra, \rs, 0, 1)$ be a bounded pseudo-hoop and let 
$\mu\in \mathcal{IS}^{(II)}_1(A)$ or $\mu\in \mathcal{IS}^{(III)}_1(A)$. Then  
$\mu(x^{-})=\mu^{-}(x)$ and $\mu(x^{\sim})=\mu^{\sim}(x)$, for all $x\in A$. \\
If $A$ is involutive, the result is also valid for $\mu\in \mathcal{IS}^{(I)}_1(A)$.
\end{prop}
\begin{proof}
If $\mu\in \mathcal{IS}^{(II)}_1(A)$, then we have: \\
$\hspace*{1cm}$ $\mu(x^{-})=\mu(x\ra 0)=\mu(0\vee_1 x)\ra \mu(0)=\mu(x)\ra 0=\mu^{-}(x)$. \\ 
Let $\mu\in \mathcal{IS}^{(III)}_1(A)$. It follows that: \\ 
$\hspace*{1cm}$ $\mu(x^{-})=\mu(x\ra 0)=\mu(x)\ra \mu(x\wedge 0)=\mu(x)\ra \mu(0)=\mu(x)\ra 0=\mu^{-}(x)$. \\
If $A$ is involutive and $\mu\in \mathcal{IS}^{(I)}_1(A)$, then we get: \\
$\hspace*{1cm}$ $\mu(x^{-})=\mu(x\ra 0)=\mu(x\vee_1 0)\ra \mu(0)=\mu(x^{-\sim})\ra 0=\mu(x)\ra 0=\mu^{-}(x)$. \\
Similarly $\mu(x^{\sim})=\mu^{\sim}(x)$.
\end{proof}

\begin{prop} \label{psh-st-60} Let $(A, \odot, \ra, \rs, \mu, 0, 1)$ be an involutive type I, type II or 
a type III state pseudo-hoop. Then $\Ker(\mu)\in {\mathcal F}_i(A)$. 
\end{prop}
\begin{proof}
It is a consequence of Proposition \ref{psh-st-50}$(6)$ and Corollary \ref{psh-if-70-20}.
\end{proof}

\begin{cor} \label{psh-st-70} If $(A, \odot, \ra, \rs, \mu, 0, 1)$ be a bounded Wajsberg type I, type II or 
a type III state pseudo-hoop, then $\Ker(\mu)\in {\mathcal F}_i(A)$. 
\end{cor}

\begin{prop} \label{psh-st-80} If $(A, \odot, \ra, \rs, \mu, 1)$ be a type II state pseudo-hoop. 
Then the following hold: \\
$(1)$ $y\le x$ implies $\mu(x\ra y)=\mu(x)\ra \mu(y)$ and $\mu(x\rs y)=\mu(x)\rs \mu(y);$ \\
$(2)$ if $A$ is bounded, then $\Ker(\mu)\in {\mathcal F}_i(A)$. 
\end{prop}
\begin{proof}
$(1)$ If $y\le x$, then by Proposition \ref{psh-80}$(2)$ we have $y\vee_1 x=y\vee_2 x=x$. Hence: \\
$\hspace*{2cm}$
$\mu(x\ra y)=\mu(y\vee_1 x)\ra \mu(y)= \mu(x)\ra \mu(y)$, \\  
$\hspace*{2cm}$
$\mu(x\rs y)=\mu(y\vee_2 x)\rs \mu(y)= \mu(x)\rs \mu(y)$. \\  
$(2)$ Let $x\in A$. Since $x\le x\vee_1 0$ and $x\le x\vee_2 0$, applying $(1)$ we get: \\
$\hspace*{2cm}$
$1=\mu(1)=\mu(0\ra x)=\mu(x\vee_1 0)\ra \mu(x)$ \\
$\hspace*{3.5cm}$
$=\mu(x\vee_1 0\ra x)=\mu(x^{-\sim}\ra x)$, \\
$\hspace*{2cm}$
$1=\mu(1)=\mu(0\rs x)=\mu(x\vee_2 0)\rs \mu(x)$ \\
$\hspace*{3.5cm}$
$=\mu(x\vee_2 0\rs x)=\mu(x^{\sim-}\rs x)$, \\
hence $x^{-\sim}\ra x, x^{\sim-}\ra x\in \Ker(\mu)$. 
It follows that $\Ker(\mu)\in {\mathcal F}_i(A)$.
\end{proof}

\begin{theo} \label{psh-st-90} For any pseudo-hoop $A$, $\mathcal{IS}^{(II)}(A)\subseteq \mathcal{IS}^{(III)}(A)$. 
\end{theo}
\begin{proof}
Let $(A, \odot, \ra, \rs, 1)$ be a pseudo-hoop, $\mu\in \mathcal{IS}^{(II)}(A)$ and $x, y\in A$. 
Since $x\wedge y\le x$, applying Proposition \ref{psh-st-80}$(1)$ we get: \\
$\hspace*{2cm}$ $\mu(x\ra y)=\mu(x)\ra \mu(x\wedge y)$ and $\mu(x\rs y)=\mu(x)\rs \mu(x\wedge y),$ \\ 
that is $(IS^{''}_1)$. Hence $\mathcal{IS}^{(II)}(A)\subseteq \mathcal{IS}^{(III)}(A)$. 
\end{proof}

\begin{theo} \label{psh-st-100} If $A$ is a bounded Wajsberg pseudo-hoop, then 
$\mathcal{IS}^{(I)}(A)\subseteq \mathcal{IS}^{(III)}(A)$.  
\end{theo}
\begin{proof}
Let $(A, \odot, \ra, \rs, 0, 1)$ be a bounded Wajsberg pseudo-hoop, $\mu\in \mathcal{IS}^{(I)}_1(A)$ and $x, y\in A$. 
Since $A$ is involutive, applying Theorem \ref{psh-waj-40} and Proposition \ref{psh-waj-20-05} we get: \\
$\hspace*{2cm}$
$\mu(x\ra y)=\mu(y^{-}\rs x^{-})=\mu(y^{-}\vee_1 x^{-})\rs \mu(x^{-})$ \\
$\hspace*{3.7cm}$
$=\mu(y^{-}\vee x^{-})\rs \mu(x^{-})=\mu^{\sim}(x^{-})\ra \mu^{\sim}(y^{-}\vee x^{-})$ \\
$\hspace*{3.7cm}$
$=\mu^{\sim}(x^{-})\ra \mu^{\sim}((x\wedge y)^{-})=\mu^{\sim-}(x)\ra \mu^{\sim-}(x\wedge y)$ \\
$\hspace*{3.7cm}$
$=\mu(x)\ra \mu(x\wedge y)$, \\
$\hspace*{2cm}$
$\mu(x\rs y)=\mu(y^{\sim}\ra x^{\sim})=\mu(y^{\sim}\vee_1 x^{\sim})\ra \mu(x^{\sim})$ \\
$\hspace*{3.7cm}$
$=\mu(y^{\sim}\vee x^{\sim})\ra \mu(x^{\sim})=\mu^{-}(x^{\sim})\rs \mu^{-}(y^{\sim}\vee x^{\sim})$ \\
$\hspace*{3.7cm}$
$=\mu^{-}(x^{\sim})\rs \mu^{-}((x\wedge y)^{\sim})=\mu^{-\sim}(x)\rs \mu^{-\sim}(x\wedge y)$ \\
$\hspace*{3.7cm}$
$=\mu(x)\rs \mu(x\wedge y)$. \\
It follows that $\mu$ satisfies $(IS^{''}_1)$, that is $\mu\in \mathcal{IS}^{(III)}(A)$. 
Hence $\mathcal{IS}^{(I)}(A)\subseteq \mathcal{IS}^{(III)}(A)$.  
\end{proof}

\begin{Def} \label{psh-st-110} An internal state $\mu$ on a pseudo-hoop $A$ is said 
to be \emph{compatible} if $\Ker(\mu)\in {\mathcal F}_n(A)$. In this case $(A,\mu)$ is said to be a 
\emph{compatible type I(type II, type III) state pseudo-hoop}. 
\end{Def}

Denote $\mathcal{IS}^{(I)}_c(A)$, $\mathcal{IS}^{(II)}_c(A)$ and $\mathcal{IS}^{(III)}_c(A)$ the set of all 
compatible internal states of type I, II and III on a pseudo-hoop $A$, respectively. 

\begin{prop} \label{psh-st-120} If $A$ is a pseudo-hoop, then $\mathcal{IS}^{(III)}_c(A)=\mathcal{IS}^{(III)}(A)$.
\end{prop} 
\begin{proof} Obviously $\mathcal{IS}^{(III)}_c(A)\subseteq \mathcal{IS}^{(III)}(A)$. 
Let $\mu\in \mathcal{IS}^{(III)}(A)$. 
By Proposition \ref{psh-st-50}$(6)$, $\Ker(\mu)\in {\mathcal F}(A)$. 
Consider $x, y\in A$ such that $x\ra y\in \Ker(\mu)$, that is $\mu(x\ra y)=1$. \\
It follows that $\mu(x)\ra \mu(x\wedge y)=1$, so $\mu(x) \le \mu(x\wedge y)$, hence $\mu(x)=\mu(x\wedge y)$. \\ 
We get $\mu(x\rs y)=\mu(x)\rs \mu(x\wedge y)=1$, that is $x\rs y\in \Ker(\mu)$. 
Similaly, from $x\rs y\in \Ker(\mu)$ we have $x\ra y\in \Ker(\mu)$. 
It follows that $\Ker(\mu)\in {\mathcal F}_n(A)$, that is $\mu\in \mathcal{IS}^{(III)}_c(A)$. \\
Hence $\mathcal{IS}^{(III)}(A)\subseteq \mathcal{IS}^{(III)}_c(A)$, and we conclude that $\mathcal{IS}^{(III)}_c(A)=\mathcal{IS}^{(III)}(A)$.
\end{proof} 

\begin{theo} \label{psh-st-130} Let $(A, \odot, \ra, \rs, \mu, 1)$ be a compatible type II pseudo-hoop. Then: \\
$(1)$ the map $\hat \mu:A/\Ker(\mu) \to A/\Ker(\mu)$ defined by $\hat \mu(x/\Ker(\mu))= \mu(x)/\Ker(\mu)$ is both 
a compatible type I and type II state operator on $A/\Ker(\mu);$ \\
$(2)$ $(A, \mu)$ is a compatible type I state pseudo-hoop.
\end{theo} 
\begin{proof}
$(1)$ If $x/\Ker(\mu)=y/\Ker(\mu)$ then $x\ra y, y\rs x\in \Ker(\mu)$, that is 
$\mu(x\ra y)=\mu(y\rs x)=1$. Applying Proposition \ref{psh-st-50}$(4)$, it follows that 
$\mu(x)=\mu(y)$. Hence $\hat \mu$ is well defined. 
The proof of the fact that $\hat \mu$ is a compatible type II state on $A/\Ker(\mu)$ is straightforward. \\
By hypothesis, $\Ker(\mu)\in {\mathcal F}_n(A)$. We show that $\Ker(\mu)\in {\mathcal F}_f(A)$. 
Indeed, let $x, y\in A$ such that $x\ra y\in \Ker(\mu)$, that is $\mu(x\ra y)=1$. 
Since $y\le (y\ra x)\rs x)$, according to Proposition \ref{psh-st-80}$(1)$ we get 
$\mu(y\vee_1 x\ra y)=\mu(y\vee_1 x)\ra \mu(y)$. 
Hence: \\
$\hspace*{2cm}$
$1=\mu(x\ra y)=\mu(y\vee_1 x)\ra \mu(y)=\mu(y\vee_1 x\ra y)$, \\ 
that is $y\vee_1 x\ra y\in \Ker(\mu)$. 
Similarly $x\rs y\in \Ker(\mu)$ implies $y\vee_2 x\rs y\in \Ker(\mu)$, that is $\Ker(\mu)\in {\mathcal F}_f(A)$. 
According to Theorem \ref{psh-ff-30}, $A/\Ker(\mu)$ is a Wajsberg pseudo-hoop. \\
By Proposition \ref{psh-st-40}, $\hat \mu$ is also a compatible type I state on $A/\Ker(\mu)$. \\
$(2)$ Let $x, y\in A$. Since $A/\Ker(\mu)$ is a Wajsberg pseudo-hoop, it follows that: \\
$\hspace*{3cm}$ $(y\vee_1 x)/\Ker(\mu)=(x\vee_1 y)/\Ker(\mu)$, \\ 
$\hspace*{3cm}$ $(y\vee_2 x)/\Ker(\mu)=(x\vee_2 y)/\Ker(\mu)$. \\
Similarly as in $(1)$ we get $\mu(y\vee_1 x)=\mu(x\vee_1 y)$ and $\mu(y\vee_2 x)=\mu(x\vee_2 y)$. 
Hence: \\
$\hspace*{3cm}$ $\mu(y\vee_1 x)\ra \mu(y)=\mu(x\vee_1 y)\ra \mu(y)$, \\ 
$\hspace*{3cm}$ $\mu(y\vee_2 x)\rs \mu(y)=\mu(x\vee_2 y)\rs \mu(y)$. \\
Thus $(A, \mu)$ is a compatible type I state pseudo-hoop.
\end{proof}

\begin{theo} \label{psh-st-140} Let $\mu:A\longrightarrow A$ be a map on a pseudo-hoop $A$. 
Then $(A, \mu)$ is a compatible type II state pseudo-hoop if and only if $(A, \mu)$ is a compatible type I 
state pseudo-hoop and $\Ker(\mu)\in {\mathcal F}_f(A)$.
\end{theo}
\begin{proof}
Suppose that $(A, \mu)$ is a compatible type II state pseudo-hoop. Then, according to Theorem \ref{psh-st-130}, 
$(A, \mu)$ is a compatible type I state pseudo-hoop. \\
As we proved in Theorem \ref{psh-st-130}, $\Ker(\mu)\in {\mathcal F}_f(A)$. \\
Conversely, let $(A, \mu)$ be a compatible type I state pseudo-hoop such that $\Ker(\mu)$ is a 
fantastic filter of $A$. According to Theorem \ref{psh-ff-30}, $A/\Ker(\mu)$ is a Wajsberg pseudo-hoop. 
It follows that: \\
$\hspace*{3cm}$ $(y\vee_1 x)/\Ker(\mu)=(x\vee_1 y)/\Ker(\mu)$ \\
$\hspace*{3cm}$ $(y\vee_2 x)/\Ker(\mu)=(x\vee_2 y)/\Ker(\mu)$, \\
for all $x, y\in A$. Similarly as in the proof of Theorem \ref{psh-st-130} we have: \\
$\hspace*{3cm}$ $\mu(y\vee_1 x)\ra \mu(y)=\mu(x\vee_1 y)\ra \mu(y)$ \\
$\hspace*{3cm}$ $\mu(y\vee_2 x)\rs \mu(y)= \mu(x\vee_2 y)\rs \mu(y)$. \\
Thus $(A, \mu)$ is a compatible type II state pseudo-hoop.
\end{proof}

\begin{ex} \label{psh-st-150} Consider the bounded hoop $(A,\odot,\ra,0,1)$ 
from Example \ref{psh-ff-70} and the maps $\mu_i:A\longrightarrow A$, $i=1,2,\cdots,7$, given in the table below:
\[
\begin{array}{c|cccccc}
 x & 0 & a & b & c & 1 \\ \hline
\mu_1(x) & 0 & 0 & 1 & 1 & 1 \\
\mu_2(x) & 0 & a & b & c & 1 \\
\mu_3(x) & 0 & a & b & 1 & 1 \\
\mu_4(x) & 0 & 1 & 0 & 1 & 1 \\
\mu_5(x) & a & a & 1 & 1 & 1 \\
\mu_6(x) & b & 1 & b & 1 & 1 \\
\mu_7(x) & 1 & 1 & 1 & 1 & 1 
\end{array}
.   
\]
Then we have: \\
$(1)$ $\mathcal{IS}^{(I)}(A)=\mathcal{IS}^{(III)}(A)=\{\mu_1,\mu_2,\mu_3,\mu_4,\mu_5,\mu_6,\mu_7\}$, \\
$(2)$ $\mathcal{IS}^{(II)}(A)=\{\mu_1,\mu_3,\mu_4,\mu_5,\mu_6,\mu_7\}$, \\
$(3)$ $\mathcal{IS}^{(I)}_1(A)=\mathcal{IS}^{(III)}_1(A)=\{\mu_1,\mu_2,\mu_3,\mu_4\}$, \\ 
$(4)$ $\mathcal{IS}^{(II)}_1(A)=\{\mu_1,\mu_3,\mu_4\}$, \\
$(5)$ $\Ker(\mu_1)=\Ker(\mu_5)=\{b,c,1\}$, $\Ker(\mu_2)=\{1\}$, 
      $\Ker(\mu_3)=\{c,1\}\in {\mathcal F}_i(A)\}$, \\
$\hspace*{2.3cm}$      
      $\Ker(\mu_4)=\Ker(\mu_6)=\{a,c,1\}$, $\Ker(\mu_7)=A$.   
\end{ex}

\begin{ex} \label{psh-st-160} Consider the bounded Wajsberg hoop $(A,\odot,\ra,0,1)$ 
from Example \ref{psh-ff-80}. 
Then we have: $\mathcal{IS}^{(I)}(A)=\mathcal{IS}^{(II)}(A)=\mathcal{IS}^{(III)}(A)=\{1_A, Id_A\}$. 
\end{ex}

$\vspace*{5mm}$

\section{State-morphism pseudo-hoops}

In this section we define the notion of a state-morphism operator on pseudo-hoops and we prove that any 
state-morphism operator is a type I and type III state operator. For the case of an idempotent pseudo-hoop 
it is proved that any type II or type III state operator is a state-morphism operator, while for a bounded 
idempotent Wajsberg pseudo-hoop any type I state operator is also a state-morphism. 
Another main result consists of proving that any state-morphism on the subalgebra of involutive elements of a 
bounded idempotent pseudo-hoop $A$ can be extended to a state-morphism on $A$. 

\begin{Def} \label{psh-sm-10} Let $(A, \odot, \ra, \rs, 1)$ be a pseudo-hoop. 
A homomorphism $\mu:A\longrightarrow A$ is called a \emph{state-morphism operator} on $A$ if $\mu^2=\mu$, where $\mu^2=\mu\circ \mu$. The pair $(A, \mu)$ is called a \emph{state-morphism pseudo-hoop}.
\end{Def}

Denote $\mathcal{SM}(A)$ the set of all state-morphism operators on a pseudo-hoop $A$. \\
If $A$ is a bounded pseudo-hoop, then denote $\mathcal{SM}_1(A)=\{\mu\in \mathcal{SM}(A) \mid \mu(0)=0\}$. 

\begin{exs} \label{psh-sm-20}
$(1)$ $1_A, Id_A\in \mathcal{SM}(A)$ for any pseudo-hoop $A$. \\
$(2)$ If $A$ is the pseudo-hoop from Example \ref{psh-st-150}, then $\mathcal{SM}(A)=\{\mu_i \mid i=1,2,\cdots, 7\}$ 
and $\mathcal{SM}_1(A)=\{\mu_i \mid i=1,2,3,4\}$. 
\end{exs}

\begin{ex} \label{psh-sm-30} Let $(A_1, \odot_1, \ra_1, \rs_1, 1_1)$ and 
$(A_2, \odot_2, \ra_2, \rs_2, 1_2)$ be two pseudo-hoops and let $A$ be the pseudo-hoop
defined in Example \ref{psh-st-30}. Then the maps $\mu_1, \mu_2:A\longrightarrow A$ defined by 
$\mu_1((x,y))=(x,x)$ and $\mu_2((x,y))=(y,y)$, for all $(x,y)\in A$ are state-morphism operators on $A$.
\end{ex}

\begin{theo} \label{psh-sm-40} For any pseudo-hoop $A$, 
$\mathcal{SM}(A)\subseteq \mathcal{IS}^{(I)}(A)\cap \mathcal{IS}^{(III)}(A)$.
\end{theo}
\begin{proof} Let $(A,\mu)$ be a state-morphism pseudo-hoop and let $x, y\in A$. 
We verify the axioms of type I and type III internal 
states. \\
$(IS_1)$ From $x\le x\vee_1 y$ we get $\mu(x)\le \mu(x\vee_1 y)$, hence 
$\mu(x\vee_1 y)\ra \mu(y)\le \mu(x)\ra \mu(y)$. 
On the other hand $\mu(x\ra y)\le (\mu(x\ra y)\rs \mu(y))\ra \mu(y)=\mu(x\vee_1 y)\ra \mu(y)$. \\ 
It follows that $\mu(x\ra y)=\mu(x\vee_1 y)\ra \mu(y)$, that is $(IS_1)$. \\
$(IS^{''}_1)$ Since $x\ra y=x\ra x\wedge y$ and $x\rs y=x\rs x\wedge y$, we get: \\
$\hspace*{2cm}$ $\mu(x\ra y)=\mu(x\ra x\wedge y)=\mu(x)\ra \mu(x\wedge y)$, \\
$\hspace*{2cm}$ $\mu(x\rs y)=\mu(x\rs x\wedge y)=\mu(x)\rs \mu(x\wedge y)$, \\ 
that is $(IS^{''}_1)$. \\
$(IS_2)$ Since $x\odot y\le x, y$ we have: \\
$\hspace*{0.5cm}$
$\mu(x)\odot \mu(x\rs x\odot y)=\mu(x)\odot (\mu(x)\rs \mu(x\odot y))= 
\mu(x)\wedge \mu(x\odot y)=\mu(x\odot y)$, \\
$\hspace*{0.5cm}$
$\mu(y\ra x\odot y)\odot \mu(y)=(\mu(y)\ra \mu(x\odot y))\odot \mu(y)=
\mu(y)\wedge \mu(x\odot y)=\mu(x\odot y)$, \\
hence $(IS_2)$ is satisfied. \\
$(IS_3)$ Since $\mu^2=\mu$, we have $\mu(\mu(x)\odot \mu(y))=\mu(\mu(x\odot y))=\mu(x\odot y)$, 
that is $(IS_3)$. \\
$(IS_4)$ Applying again the property $\mu^2=\mu$ we get: \\
$\hspace*{2cm}$ $\mu(\mu(x)\ra \mu(y))=\mu(\mu(x\ra y))=\mu(x\ra y)$, \\
$\hspace*{2cm}$ $\mu(\mu(x)\rs \mu(y))=\mu(\mu(x\rs y))=\mu(x\rs y)$, \\
thus $(IS_4)$ is verified. 
It follows that $\mu \in \mathcal{IS}^{(I)}(A)$ and $\mu \in \mathcal{IS}^{(III)}(A)$, so that \\
$\mathcal{SM}(A)\subseteq \mathcal{IS}^{(I)}(A)\cap \mathcal{IS}^{(III)}(A)$.
\end{proof}

\begin{cor} \label{psh-sm-50} If $A$ is a Wajsberg pseudo-hoop, then  
$\mathcal{SM}(A)\subseteq \mathcal{IS}^{(I)}(A)\cap \mathcal{IS}^{(II)}(A)\cap \mathcal{IS}^{(III)}(A)$.
\end{cor}
\begin{proof} It follows from Theorem \ref{psh-sm-40} and Proposition \ref{psh-st-40}.
\end{proof}

\begin{theo} \label{psh-sm-50-10} $\rm($\cite[Th. 3.17]{Ciu22}$\rm)$ If $A$ is an idempotent pseudo-hoop, then  
$\mathcal{IS}^{(III)}(A)\subseteq \mathcal{SM}(A)$. 
\end{theo}

\begin{cor} \label{psh-sm-50-20} For any idempotent pseudo-hoop $A$, 
$\mathcal{IS}^{(II)}(A)\subseteq \mathcal{SM}(A)$. 
\end{cor}
\begin{proof} It follows by Theorems \ref{psh-st-90} and \ref{psh-sm-50-10}. 
\end{proof}

\begin{cor} \label{psh-sm-50-30} For any bounded idempotent Wajsberg pseudo-hoop $A$, 
$\mathcal{IS}^{(I)}(A)\subseteq \mathcal{SM}(A)$. 
\end{cor}
\begin{proof} It follows by Theorems \ref{psh-st-100} and \ref{psh-sm-50-10}. 
\end{proof}

\begin{lemma} \label{psh-sm-50-40} Let $A$ be a bounded pseudo-hoop and let $\mu\in \mathcal{SM}_1(A)$. 
Then the following hold: \\
$(1)$ $\mu(x^{-})=\mu^{-}(x)$ and $\mu(x^{\sim})=\mu^{\sim}(x)$, for all $x\in A;$ \\
$(2)$ if $x\in \Inv(A)$, then $\mu(x)\in \Inv(A)$. 
\end{lemma}
\begin{proof} It is straightforward.
\end{proof}

\begin{theo} \label{psh-sm-60} Let $A$ be a normal good pseudo-hoop and let $\mu\in \mathcal{SM}_1(\Inv(A))$. 
If $\tilde{\mu}:A\longrightarrow A$ defined by $\tilde{\mu}(x)=\mu(x^{-\sim})$, for all $x\in A$,  
then $\tilde{\mu}\in \mathcal{SM}_1(A)$ such that $\tilde \mu_{\mid \Inv(A)}=\mu$. 
\end{theo}
\begin{proof} Obviously $\tilde{\mu}(0)=\mu(0^{-\sim})=\mu(0)=0$ and $\tilde{\mu}(1)=\mu(1^{-\sim})=\mu(1)=1$. \\ 
Applying Proposition \ref{psh-30} we get: \\ 
$\hspace*{2cm}$ $\tilde{\mu}(x\ra y)=\mu((x\ra y)^{-\sim})=\mu(x^{-\sim}\ra y^{-\sim})$ \\
$\hspace*{3.7cm}$ $=\mu(x^{-\sim})\ra \mu(y^{-\sim})=\tilde{\mu}(x)\ra \tilde{\mu}(y)$, \\
$\hspace*{2cm}$ $\tilde{\mu}(x\rs y)=\mu((x\rs y)^{-\sim})=\mu(x^{-\sim}\rs y^{-\sim})$ \\
$\hspace*{3.7cm}$ $=\mu(x^{-\sim})\rs \mu(y^{-\sim})=\tilde{\mu}(x)\rs \tilde{\mu}(y)$, \\
$\hspace*{2cm}$ $\tilde{\mu}(x\wedge y)=\mu((x\wedge y)^{-\sim})=\mu((x^{-\sim}\wedge y^{-\sim}))$ \\
$\hspace*{3.5cm}$ $=\mu(x^{-\sim})\wedge \mu(y^{-\sim})=\tilde{\mu}(x)\wedge \tilde{\mu}(y)$, \\
$\hspace*{2cm}$ $\tilde{\mu}(x\odot y)=\mu((x\odot y)^{-\sim})=\mu(x^{-\sim}\odot y^{-\sim})$ \\
$\hspace*{3.5cm}$ $=\mu(x^{-\sim})\odot \mu(y^{-\sim})=\tilde{\mu}(x)\odot \tilde{\mu}(y)$. \\ 
Since $\tilde{\mu}^2(x)=\mu^2(x^{-\sim})=\mu(x^{-\sim})=\tilde{\mu}(x)$, we conclude that 
$\tilde{\mu}\in \mathcal{SM}_1(A)$. \\
Moreover, if $x\in \Inv(A)$, then $\tilde{\mu}(x)=\mu(x^{-\sim})=\mu(x)$, that is $\tilde \mu_{\mid \Inv(A)}=\mu$. 
\end{proof}

\begin{cor} \label{psh-sm-70} If $A$ is a bounded idempotent pseudo-hoop, then any state-morphism on $\Inv(A)$ 
can be extended to a state-morphism on $A$.
\end{cor}
\begin{proof}
According to Proposition \ref{psh-20-05}, $x\ra y=x\rs y$, for all $x, y\in A$, hence $A$ is good. \\
The assertion follows from Remark \ref{psh-waj-80} and Theorem \ref{psh-sm-60}.
\end{proof}

\begin{theo} \label{psh-sm-80} Let $A$ be a bounded pseudo-hoop, $s\in \mathcal{BS}(A)$, $\mu\in \mathcal{SM}_1(A)$  
and $s_{\mu}:A\longrightarrow [0, 1]$, defined by $s_{\mu}(x)=s(\mu(x))$, for all $x\in A$.  
Then $s_{\mu}\in \mathcal{BS}(A)$. 
\end{theo}
\begin{proof}
Obviously $s_{\mu}(0)=0$ and $s_{\mu}(1)=1$. For all $x, y\in A$ we have: \\
$\hspace*{2cm}$ $s_{\mu}(x)+s_{\mu}(x\ra y)=s(\mu(x))+s(\mu(x\ra y))=s(\mu(x))+s(\mu(x)\ra \mu(y))$ \\
$\hspace*{5.2cm}$ $=s(\mu(y))+s(\mu(y)\ra \mu(x))=s(\mu(y))+s(\mu(y\ra x))$ \\
$\hspace*{5.2cm}$ $=s_{\mu}(y)+s_{\mu}(y\ra x)$. \\
Similarly $s_{\mu}(x)+s_{\mu}(x\rs y)=s_{\mu}(y)+s_{\mu}(y\rs x)$, hence $s_{\mu}\in \mathcal{BS}(A)$. 
\end{proof}

\begin{theo} \label{psh-sm-90} Let $A$ be a good pseudo-hoop, $s\in \mathcal{BS}(\Inv(A))$, $\mu\in \mathcal{SM}_1(A)$  
and $\tilde{s}_{\mu}:A\longrightarrow [0, 1]$, defined by $\tilde{s}_{\mu}(x)=s(\mu(x^{-\sim}))$, for all $x\in A$.  
Then $\tilde{s}_{\mu}\in \mathcal{BS}(A)$. 
\end{theo}
\begin{proof}
Obviously $\tilde{s}_{\mu}(0)=0$ and $\tilde{s}_{\mu}(1)=1$. If $x, y\in A$, applying Proposition \ref{psh-30} 
we get: \\
$\hspace*{2cm}$ $\tilde{s}_{\mu}(x)+\tilde{s}_{\mu}(x\ra y)=s(\mu(x^{-\sim}))+s(\mu((x\ra y)^{-\sim}))$ \\
$\hspace*{5.2cm}$ $=s(\mu(x^{-\sim}))+s(\mu(x^{-\sim}\ra y^{-\sim}))$ \\
$\hspace*{5.2cm}$ $=s(\mu(x^{-\sim}))+s(\mu(x^{-\sim})\ra \mu(y^{-\sim}))$ \\
$\hspace*{5.2cm}$ $=s(\mu(y^{-\sim}))+s(\mu(y^{-\sim})\ra \mu(x^{-\sim}))$ \\
$\hspace*{5.2cm}$ $=s(\mu(y^{-\sim}))+s(\mu(y^{-\sim}\ra x^{-\sim}))$ \\
$\hspace*{5.2cm}$ $=s(\mu(y^{-\sim}))+s(\mu((y\ra x)^{-\sim}))$ \\
$\hspace*{5.2cm}$ $=\tilde{s}_{\mu}(y)+\tilde{s}_{\mu}(y\ra x)$. \\
Similarly $\tilde{s}_{\mu}(x)+\tilde{s}_{\mu}(x\rs y)=\tilde{s}_{\mu}(y)+\tilde{s}_{\mu}(y\rs x)$, 
hence $\tilde{s}_{\mu}\in \mathcal{BS}(A)$. 
\end{proof}

\begin{prop} \label{psh-sm-100} Let $A$ be a pseudo-hoop and $\mu\in \mathcal{SM}(A)$. Then the following hold: \\
$(1)$ $\mu$ is injective iff $\Ker(\mu)=\{1\};$ \\
$(2)$ if $F\in \mathcal{F}(A)$, then $\mu^{-1}(F)\in \mathcal{F}(A);$ \\
$(3)$ if $F\in \mathcal{F}_n(A)$, then $\mu^{-1}(F)\in \mathcal{F}_n(A);$ \\
$(4)$ if $A$ is bounded, $\mu\in \mathcal{SM}_1(A)$ and $F\in \mathcal{F}_i(A)$, 
then $\mu^{-1}(F)\in \mathcal{F}_i(A)$. 
\end{prop}
\begin{proof}
$(1)-(3)$ See \cite[Th. 6.12]{Ciu21}. \\
$(4)$ It follows by Proposition \ref{psh-if-100}. 
%Let $x\in A$. Since $F\in \mathcal{F}_i(A)$, we have $\mu^{-\sim}(x)\ra \mu(x)\in F$. \\
%By Lemma \ref{psh-sm-50-40}, $\mu(x^{-\sim})\ra \mu(x)\in F$, so $\mu(x^{-\sim}\ra x)\in F$. 
%It follows that $x^{-\sim}\ra x\in \mu^{-1}(F)$. \\
%Similarly $x^{\sim-}\ra x\in \mu^{-1}(F)$, hence $\mu^{-1}(F)\in \mathcal{F}_i(A)$.
\end{proof}

$\vspace*{5mm}$

\section{Concluding remarks}

Developing probabilistic theories on algebras of fuzzy logics is a central topic for the study of fuzzy systems. 
For this purpose, different probabilistic models have been constructed on algebras of multiple-valued logics: 
states, generalized states, internal states, state-morphism operators, measures. 
Probabilistic models on hoops and pseudo-hoops were topics of many works (\cite{Bor2}, \cite{Ciu10}, \cite{Ciu20}, \cite{Ciu22}, \cite{Ciu23}). 
In this paper we show that the particular case of Wajsberg pseudo-hoops and the involutive filters play an important role in probabilities theory on pseudo-hoops. We unified different concepts of internal states on pseudo-hoops and 
proved that in the case of Wajsberg pseudo-hoops the three types of internal states coincide. 
Important results on probabilistic models on algebras of non-classical logic have been proved based on involutive filters. \\
We suggest further directions of research, as the above topics are of current interest. \\
For the case of state R$\ell$-monoids $(M,\mu)$ the notion of $\mu$-filters was introduced in \cite{DvRa3}. 
One can define the notion of $\mu$-filters on pseudo-hoops and investigate the correspondence between the existence 
of state operators and the maximal and normal $\mu$-filters on state pseudo-hoops. 
The notion of involutive $\mu$-filter on pseudo-hoops could be an interesting topic of research. \\
Subdirectly irreducible state R$\ell$-monoids have been introduced and studied in \cite{DvRa3} and \cite{DvRa4}. 
As a further research topic one could define and investigate the irreducible state pseudo-hoops.

%\begin{center}
%\sc Acknowledgement 
%\end{center}
%The author is very grateful to the anonimous referees for their useful remarks and suggestions on the subject that %helped improving the presentation.

$\vspace*{5mm}$

\setlength{\parindent}{0pt}

\vspace*{3mm}
\begin{flushright}
\begin{minipage}{148mm}\sc\footnotesize
Lavinia Corina Ciungu\\
Department of Mathematics \\
University of Iowa \\
14 MacLean Hall, Iowa City, Iowa 52242-1419, USA \\
{\it E--mail address}: {\tt lavinia-ciungu@uiowa.edu}

\vspace*{3mm}

\end{minipage}
\end{flushright}

\end{document}